\documentclass[english,11pt,twoside]{cwarticle}
\usepackage{perso-math}
\allowdisplaybreaks
\usepackage{tikz}
	\usetikzlibrary{arrows, hobby, quantikz2}
	\definecolor{colorAbsorbing}{RGB}{164, 81, 60}
	\definecolor{colorDelta}{RGB}{255, 205, 204}
	\definecolor{colorDeltaBound}{RGB}{195, 148, 147}
	\definecolor{colorDomain}{RGB}{50, 204, 212}
	\definecolor{colorDomainBound}{RGB}{0, 118, 125}
	\definecolor{colorInitialPoint}{RGB}{116, 54, 73}
\NewDocumentCommand{\skd}{}{Skorohod\xspace}
\DeclarePairedDelimiterX{\scalarproduct}[1]{\langle}{\rangle}{
	\IfBlankTF{#1}{\: .\:,.\: }{#1}%
}
\DeclarePairedDelimiterX{\floor}[1]{\lfloor}{\rfloor}{
	\IfBlankTF{#1}{\: .\:}{#1}%
}
\NewDocumentCommand{\norm}{om}{%
	\mleft\|%
	\IfBlankTF{#2}{\:.\:}{#2}
	\mright\|\IfNoValueTF{#1}{}{_{#1}}
}
\NewDocumentCommand{\abs}{om}{%
	\mleft|%
	\IfBlankTF{#2}{\:.\:}{#2}
	\mright|\IfNoValueTF{#1}{}{_{#1}}
}
\DeclarePairedDelimiterX{\range}[1]{\llbracket}{\rrbracket}{
	\IfBlankTF{#1}{\: .\:,.\: }{#1}%
}
\NewDocumentCommand{\set}{mo}{\left\{#1 \IfNoValueTF{#2}{}{,~#2}\right\}}

\NewDocumentCommand{\nintegers}{}{\dsN}
\NewDocumentCommand{\reals}{}{\dsR}
\NewDocumentCommand{\rationals}{}{\dsQ}

\RenewDocumentCommand{\leq}{}{\leqslant}
\RenewDocumentCommand{\geq}{}{\geqslant}

\NewDocumentCommand{\continuous}{omo}{%
	\calC%
    \IfNoValueTF{#1}{}{^{#1}}%
    \IfNoValueTF{#3}{}{_{#3}}%
    \IfBlankTF{#2}{}{\!\br{#2}}
}

\NewDocumentCommand{\borelfun}{m}{\calB_b\br{#1}}

\DeclareMathOperator*{\argmin}{arg\,min}
\NewDocumentCommand{\domain}{m}{\calD(#1)}
\DeclareMathOperator{\diag}{diag}
\NewDocumentCommand{\fun}{ommoo}{%
    \IfNoValueTF{#1}{}{#1:}%
    \IfNoValueTF{#4}{#2\longmapsto #3}{
        \begin{array}{ccl}%
            #4 & \longrightarrow & \IfNoValueTF{#5}{#4}{#5} \\ 
            #2 & \longmapsto & #3 \\
        \end{array}%
    } 
}
\DeclareMathOperator{\card}{card}

\NewDocumentCommand{\closure}{m}{\overline{#1}}

\NewDocumentCommand{\1}{}{\mathds{1}}
\NewDocumentCommand{\proba}{omo}{\bfP\IfNoValueF{#1}{_{#1}}\IfBlankTF{#2}{}{\IfNoValueTF{#3}{\br{#2}}{\br{#2\mid#3}}}}
\NewDocumentCommand{\expect}{omo}{%
    \bfE\IfNoValueF{#1}{_{#1}}\left[#2\IfNoValueF{#3}{\mid #3}\right]
    }

\NewDocumentCommand{\chain}{o}{Y\IfNoValueF{#1}{_{#1}}}
\NewDocumentCommand{\chainWF}{o}{Y^{\mathsf{WF}}\IfNoValueF{#1}{_{#1}}}
\NewDocumentCommand{\exchange}{O{N}}{\Phi_{#1}}
\NewDocumentCommand{\bernstein}{O{N}}{B_{#1}}
\NewDocumentCommand{\transition}{O{N}}{Q_{#1}}

\NewDocumentCommand{\process}{o}{X\IfNoValueF{#1}{^{#1}}}
\NewDocumentCommand{\brownian}{}{B}
\NewDocumentCommand{\browniancombi}{}{W}
\NewDocumentCommand{\processcombi}{}{R}

\NewDocumentCommand{\densityspace}{O{N}}{K_{#1}}
\NewDocumentCommand{\hypercube}{}{K}
\NewDocumentCommand{\absorbzone}{O{\alpha}}{\Delta_{#1}}
\NewDocumentCommand{\absorbingstates}{}{\calP}
\NewDocumentCommand{\indices}{o}{\calI\IfNoValueF{#1}{_{#1}}}
\NewDocumentCommand{\face}{}{\calF}
\NewDocumentCommand{\normal}{o}{\overline{n}\IfNoValueF{#1}{_{#1}}}
\NewDocumentCommand{\zero}{}{0_m}
\NewDocumentCommand{\one}{}{1_m}

\NewDocumentCommand{\boundedcube}{}{\calB_b(\hypercube)}
\NewDocumentCommand{\continuouscube}{o}{\IfNoValueTF{#1}{\continuous{\hypercube}}{\continuous[#1]{\hypercube}}}
\NewDocumentCommand{\polynomials}{o}{\calP\IfNoValueF{#1}{_{#1}}(\hypercube)}
\NewDocumentCommand{\rcll}{}{D(\reals_+,\hypercube)}
\NewDocumentCommand{\drift}{o}{b\IfNoValueF{#1}{_{#1}}}
\NewDocumentCommand{\diffusion}{o}{a\IfNoValueF{#1}{_{#1}}}
\NewDocumentCommand{\dispersion}{}{\sigma}

\NewDocumentCommand{\hittingtime}{m}{T_{#1}}

\NewDocumentCommand{\normC}{om}{\norm[\IfNoValueTF{#1}{\hypercube}{\continuous[#1]{\hypercube}}]{#2}}
\NewDocumentCommand{\seminormC}{mm}{\abs[\continuous[#1]{\hypercube}]{#2}}

\NewDocumentCommand{\operator}{}{C}
\NewDocumentCommand{\diffop}{}{A}
\NewDocumentCommand{\driftop}{}{B}
\NewDocumentCommand{\diffsemigroup}{}{U}
\NewDocumentCommand{\driftsemigroup}{}{V}
\NewDocumentCommand{\semigroup}{o}{S\IfNoValueF{#1}{_{#1}}}
\NewDocumentCommand{\generator}{o}{L\IfNoValueF{#1}{_{#1}}}

\NewDocumentCommand{\distortionprod}{}{\pi_d}
\NewDocumentCommand{\distortion}{o}{d\IfNoValueF{#1}{_{#1}}}
\NewDocumentCommand{\distortioncombi}{}{\varphi_{d}}

\usepackage[
backend=biber,
style=alphabetic, 
]{biblatex}
\title{Feller Property and Absorption of Diffusions for Multi-Species Metacommunities}
\author{%
    Benoît \textsc{Henry}% 
    \thanks{
        Univ. Lille, CNRS, IMT Nord Europe, Inria, UMR 8524 -- Laboratoire Paul Painlevé, F-59000 Lille, France.\\
        E-mail address: \href{mailto:benoit.henry@imt-nord-europe.fr}%
    {benoit.henry@imt-nord-europe.fr}%
    }%
    \and% 
    Céline \textsc{Wang}%
    \thanks{%
        Univ. Lille, CNRS, Inria PARADYSE, UMR 8524 -- Laboratoire Paul Painlevé, F-59000 Lille, France.\\
        E-mail address: \href{mailto:celine.wang@univ-lille.fr}%
        {celine.wang@univ-lille.fr}
    }
}

\begin{document}
    \maketitle 
    \begin{abstract}[metacommunity, Wright-Fisher model, Feller process, absorption]
        We consider individuals of two species distributed over $m$ \emph{patches}, each with a hosting capacity $\distortion[i]N$, where $\distortion[i]\in (0,1]$. We assume that all the patches are linked by the \emph{dispersal} of individuals. This work examines how the metacommunity evolves in these patches. The model incorporates Wright-Fisher intra-patch reproduction and a general exchange function representing dispersal. Under minimal assumptions, we demonstrate that as $N$ approaches infinity, the processes converge to a diffusion process for which we establish the Feller property. We prove that the limiting process almost surely reaches the absorbing states in finite time.
    \end{abstract} 
                  
\section*{Introduction}

    Community ecology is the study of the \emph{local dynamics} of interacting species within a given habitat. It takes into account local factors such as environmental conditions, competition and predation, which are assumed to be homogeneous. The individuals of a single species in a given area are called a \emph{population}, and all the interacting individuals within it are referred to as a \emph{local community}. Local communities can be linked by \emph{dispersal}, \ie immigration and emigration between areas. The set of interacting communities is referred to as a \emph{metacommunity}, and the populations of a single species within the metacommunity are referred to as a \emph{metapopulation}. Metacommunity theory aims to capture the dynamics at a regional scale, where species are distributed across several linked areas of habitat, known as \emph{patches}, which may have distinct local characteristics. Metacommunity frameworks are well suited to modelling organisms such as microbiomes \autocite{Johnson2025_InteractingHostsMicrobiomeExchange}, terrestrial plant species and phytoplankton species \autocite{PicocheBarraquand2022_SeedBanksPhytoplankton}.
    
    The Wright-Fisher model is a classic model of population genetics used to describe genetic variation (or evolutionary types) in populations of organisms with discrete, non-overlapping generations (such as annual plants). The associated processes are Markov processes: the individuals in a generation are offspring of parents chosen uniformly and independently at random from the immediate preceding generation. The original model incorporates two evolutionary factors that influence the dynamics: mutation and selection (for a detailed description of the model, refer to \autocite[Section 1 of Chapter 10]{ethierkurtz:86:markov}). The Wright-Fisher model and its diffusion approximations have been generalised and studied in many ways, taking into account other factors influencing population dynamics or incorporating characteristics of certain organisms. In biology, the model can describe the dynamics of species competing for resources. For instance, in articles such as \autocite{gonzalezcasanova2020} and \autocite{alsmeyer2026}, the authors propose extended models in which each individual type requires a specific amount of resources to produce offspring. Rather than fixing the number of individuals, they fix the number of resources. Another extension of interest for modelling competing species with large offspring variance, such as certain marine species, is the $\Lambda$-Wright-Fisher diffusion, which is developed, for example, in \autocite{blancas2023}. In ecology, models based on the Wright-Fisher model are used to describe plant metapopulations with seed banks. For example, in \autocite{kaj2001}, \autocite{blathkurt2021} and \autocite{louvet2022}, the parent of an individual is chosen at random from a previous generation rather than from the generation immediately preceding it.

    In the original model with two alleles or types, assuming that each generation has a constant lifespan equal to the inverse of the (adult) population size, the Wright-Fisher process (which gives the frequency of an allele or type within the population) can be approximated by a diffusion process in large populations. This diffusion process is a solution to a Stochastic Differential Equation (SDE) of the form
    \[
        \d Y_t = \sigma\sqrt{Y_t\br{1-Y_t}}\d B_t + \br{\alpha-\beta Y_t}\d t,
    \]
    where $\alpha$, $\beta$ and $\sigma$ are positive real numbers with $\alpha\geq \beta$, and $B$ is a standard Brownian motion. It is standard practice to extend the coefficients to Hölder and Lipschitz continuous functions in $\reals$ to solve the SDE. Using Yamada functions (introduced in \autocite{YamadaWatanabe1971}), it is possible to prove that the solution to the SDE with extended coefficients is $[0,1]$-valued and to establish strong uniqueness (see, for example, \autocite[184-185]{alfonsi2015affine}). Feller's test for explosions (\cf \autocite[342-350]{KaratzasShreve1991} and \autocite[Chapter 8]{ethierkurtz:86:markov}) enables us to determine the nature and accessibility of the boundary of $[0,1]$, as well as the domain of its infinitesimal generator. For the study of populations with at least $3$ alleles or types, the diffusion process is defined as a solution to a SDE for which the construction is not straightforward, as there is no general method to establish weak uniqueness for a SDE, nor are there general results explicitly determining the domain of the infinitesimal generator. In \autocite{Ethier1976}, the author establishes the weak uniqueness of the SDE and proves that the infinitesimal generator of the diffusion process is given by the closure of the degenerate differential operator.
   
    In this work, we use the Wright-Fisher-based metacommunity model introduced in \autocite{delvoye} as our starting point. The original model considers a metacommunity of two species across two patches, featuring independent intra-patch neutral Wright-Fisher reproduction and linear inter-patch dispersal of individuals. We generalise the dispersal by considering $m$ patches and an exchange function $\exchange$, which is subject only to the constraints that metapopulation sizes are conserved, and that $N\br{\exchange-I}$ can be approached by a regular function, denoted by $\drift$, when $N$ is large. The processes that we study are Feller processes on $\hypercube={[0,1]}^m$. We observe that for any $f\in\continuouscube[2]$, the functions $\generator[N]f$ (where $\generator[N]$ is the infinitesimal generator of the process with population sizes $d_1N,\dots, d_mN$) converge to $\generator f$, which is defined by
    \[
        \generator f = \frac{1}{2}\sum_{i=1}^m \frac{x_i\br{1-x_i}}{d_i}\pd{f}{x_i,x_i}+\sum_{i=1}^m \drift[i](x)\pd{f}{x_i},~\domain{\generator} = \continuouscube[2]
    \]
    (see \zcref{sec:convergence} for the details). By adapting the proofs in \autocite{Ethier1976}, we demonstrate that its closure $\closure{\generator}$ generates a Feller semigroup. In particular, $\continuouscube[2]$ will be a core for the infinitesimal generator. This result is interesting for two reasons. Firstly, the infinitesimal generator can be characterised without specifying boundary conditions. Secondly, this fact is then exploited to prove the convergence of the discrete model to its diffusive approximation.
        
    Due to the conservation property of the model, if a species disappears at a given time, it will not reappear. Thus, $\zero = \br{0,\dots, 0}$ and $\one = \br{1,\dots, 1}$ are absorbing for the diffusion process. One natural question that arises is whether a species will go extinct, \ie whether the process will hit $\absorbingstates=\set{\zero, \one}$ within a finite time. In the context of multidimensional Wright-Fisher models, these questions have already been investigated. In models involving mutations, a given trait can reappear, thereby preventing extinction and leading to the existence of a stationary distribution (see \eg \autocite{jenkins2017exact}). Without mutation, the question has essentially been settled in \autocite{Ethier79Absorption}. However, in this case, the $m-1$-dimensional faces of the simplex are absorbing, and the problem sequentially decreases in dimension as long as lower dimensional faces are reached by the process, ultimately returning to the well-known one-dimensional problem. Our case is substantially more difficult: since a species can always reappear through immigration if it is absent from a patch, extinction occurs if and only if the diffusion exactly hits the particular points $\zero$ or $\one$ of the hypercube. The attainability of a corner by a multidimensional diffusive process is a difficult problem. Indeed, as pointed out by \autocite{ernst2021escape}, even for Brownian motion in $\reals_+^2$ (or in $[0,1]^m$), whether or not $(0,0)$ is reached in finite time depends on the boundary conditions used to define the process. One standard approach for hitting problems is to look at the associated Dirichlet Problem. When the infinitesimal generator of the process is a uniformly elliptic operator $\tilde{\generator}$ on an open bounded $\continuous[2]{}$ domain $D$, denoting by $\hittingtime{\partial D}$ the hitting time of the boundary, the function $\fun[h]{x}{\proba[x]{\hittingtime{\partial D}<\infty}}$ is the unique smooth solution to the Dirichlet Problem
    \[
        \begin{cases}
            \tilde{\generator} u  = 0 &\text{ in } D\\ 
            u_{|\partial D}   = 1, 
        \end{cases}
    \]
   and the maximum principle tells us that $h$ is constant and is equal to $1$. In our case, for the hitting time $\hittingtime{\absorbingstates}$ of $\absorbingstates$, we would expect the function $\fun[h]{x}{\proba[x]{\hittingtime{\absorbingstates}<\infty}}$ to be a solution to the Problem 
   \begin{equation}
   \zcsetup{reftype=dproblem}
   \label{eq:dirichlet-problem}
        \begin{cases}
            \generator u = 0  &\text{ in } \hypercube\\ 
            u_{|\absorbingstates}   = 1. 
        \end{cases}
   \end{equation}
   However, the operator $\generator$ degenerates on the entire boundary. To adapt this approach, we should consider an analogous boundary value problem of the form 
   \[
        \begin{cases}
            \generator u  = 0  &\text{ in }\hypercube\\
            u_{|\Sigma} = 1, 
        \end{cases}
   \]
    where $\Sigma$ is the subset of $\partial \hypercube$ where the Fichera function $\fun[\psi]{x}{\sum_{i=1}^m\br{\drift[i](x)-\frac{1-2x_i}{\distortion[i]}}}\normal[i](x)$ (with $\normal=\br{\normal[1],\dots,\normal[m]}$ associated with $\generator$, the inward normal vector to the boundary $\partial \hypercube$) is negative (we refer to \autocite{oleinik2012second} for the boundary value problems for degenerate operators and Fichera functions). However, establishing the existence of a solution to this problem is difficult, and the maximum principle is not applicable in such a general case. Moreover, for many choices of $\drift$, the Fichera's boundary $\Sigma$ equals $\partial K$, which entails that the \zcref{eq:dirichlet-problem} is not even well-posed. 

    The paper is organised as follows. In \zcref{sec:model}, we provide a detailed description of our model. Based on minimal assumptions regarding the exchange functions, we then state our main results about the diffusion approximation. The subsequent sections are dedicated to the proofs of these results.

\section{The Model and Main Results}
\zcsetup{reftype=section}
\label{sec:model}
    \begin{notation}
        \zcsetup{reftype=notation}
    	\label{not:cube}
    	We denote by $\hypercube$ the compact subset ${[0,1]}^m$ of $\reals^m$. An element $x\in\hypercube$ belongs to the boundary if there exists $i\in\range{1,m}$ such that $x_i\in\set{0,1}$. For each $x\in\hypercube$, we write the set of indices $i$ for which $x_i=0$ as $\indices[0](x)$, those for which $x_i=1$ as $\indices[1](x)$, and the union of these two sets $\indices[0](x)\cup\indices[1](x)$ as $\indices(x)$.
    	We denote by $\normal = \br{\normal[1],\dots,\normal[m]}$ the inward orthogonal vector to the boundary $\partial \hypercube$ of $\hypercube$, defined by 
    	\[
    	\normal[i](x) = 
    	\begin{cases}
    		1 & \text{if } i\in\indices[0](x)\\ 
    		-1 & \text{if } i\in\indices[1](x)\\
    		0 & \text{otherwise}
    	\end{cases}
    	\]
    	for any $x\in\partial \hypercube$. We write $\zero$ for $(0,\dots,0)$ and $\one$ for $(1,\dots, 1)$. The set $\set{\zero,\one}$ is referred to as $\absorbingstates$.
    \end{notation}
        
    \subsection{Finite Population Model}
    \zcsetup{reftype=subsection}
    \label{subsec:finite-pop-model}
        Let $N$ be a positive integer. We consider a metacommunity comprising two species, $\alpha$ and $\beta$, distributed across $m$ patches hosting $N_1 = \distortion[1] N$,..., $N_{m-1}  = \distortion[m-1]$ and $N_m = \distortion[m] N$ individuals respectively, where $\distortion[1],\dots, \distortion[m-1]$ and $\distortion[m]$ are positive real numbers such that $\distortion[m]\leq \distortion[m-1]\leq \cdots \leq \distortion[1] = 1$, which are called \emph{distortions}. We assume that the generations do not overlap. Individuals of each generation migrate between patches and reproduce simultaneously. For each $k\in\nintegers$, we denote by $\chain[k] = \br{\chain[k]^{(1)},\dots, \chain[k]^{(m)}}$ the $m$-tuple of population densities of species $\alpha$ (\ie the ratio of the $\alpha$ population to the size of the patch) in each patch just before the reproduction phase, constituting generation $k+1$.
        \begin{itemize}
            \item The reproduction phase that constitutes a new generation $k+1$ from the generation $k$ is described by the neutral Wright-Fisher model independently in each patch. If $\br{\chainWF[k+1]}^{(i)}$ denotes the population density of species $\alpha$ in the $i$-th patch after reproduction, then the transition distribution is given by 
            \[
                \proba{{\br{\chainWF[k+1]}}^{(i)} = \frac{\ell}{N_i}}[\chain[k]^{(i)}=x] = \binom{N_i}{k}x^{\ell}(1-x)^{N_i-\ell},~\ell\in\range{0,N_i},~x\in\hypercube.
            \]
            \item The migration phase for generation $k+1$ is defined by a deterministic transformation $\exchange:\hypercube\rightarrow\hypercube$. The $m$-tuple $\chainWF[k+1]= \br{{\br{\chainWF[k+1]}}^{(1)},\dots, {\br{\chainWF[k+1]}}^{(m)}}$ becomes after migration
            \[
                \chain[k+1] = \exchange\br{\chainWF[k+1]}.
            \]
        \end{itemize}
        The sequence $\br{\chain[k]}$ is thereby a Markov chain with state-space $\hypercube$, and whose transition operator $\transition$ is given by
        \[
            \transition f(x) = \expect{f(\chain[k+1])}[\chain[k]=x]  = \bernstein \br{f\circ\exchange}(x)
        \]
        for any $x\in\hypercube$ and function $\fun[f]{\hypercube}{\reals}$, $\bernstein$ denoting the Bernstein operator defined by
        \[
            \bernstein f (x) = \sum_{k_1 = 0}^{N_1}\cdots \sum_{k_m=0}^{N_m} \binom{N_1}{k_1}\cdots \binom{N_m}{k_m}f\br{\frac{k_1}{N_1},\dots, \frac{k_m}{N_m}}x_1^{k_1}(1-x_1)^{N_1-k_1}\cdots x_m^{k_m}(1-x_m)^{N_m-k_m}
        \]
        for each $f\in\continuouscube$ and $x\in \hypercube$. 
        
        Let $\process[N] = \set{\process[N]_t}[t\geq 0]$ be the jump Markov process on $\hypercube$ with transition rate $N$, induced by the Markov chain $\br{\chain[k]}$. Then, its semigroup $\set{\semigroup[N](t)}[t\geq 0]$ is given by
        \[
            \semigroup[N](t) = e^{-Nt}\sum_{k=0}^{\infty}\frac{\br{Nt}^k}{k!}\transition^k,
        \]
        for any $t\in\reals_+$.  Since $\transition$ is a bounded operator on $\continuouscube$, $\set{\semigroup[N](t)}[t\geq 0]$ is a Feller semigroup, and its infinitesimal generator is $\generator[N] = N\br{\transition  - I}$.
        
    \subsection{Hypotheses on the Exchange Transformations and the Drift Coefficients}
    \label{subsec:hypotheses-exchange}
        We make the following assumptions: 
        \begin{hypothesis}
            \zcsetup{reftype=hypothesis}
        	\label{hyp:conservation}
        	The migration maintains the metapopulation size, \ie the exchange transformation $\exchange = \br{\exchange^{(1)},\dots, \exchange^{(m)}}$ satisfies for each $x\in\hypercube$,
        	\[
        	    N_1\exchange^{(1)}\br{x}+\cdots+N_m \exchange^{(m)}\br{x} = N_1x_1+\cdots+N_mx_m.
        	\]
        \end{hypothesis}
        \begin{hypothesis}
            \zcsetup{reftype=hypothesis}
            \label{hyp:existence-b}
             We assume that 
             \[
                \lim_{N\to\infty} \sup_{x\in \hypercube}\abs{\exchange(x)-x} = 0,
             \]
             and that there exists a function $\drift = \br{\drift[1],\dots, \drift[m]}\in{\continuous[2]{\hypercube,\reals^m}}$ such that
             \[
             	\lim_{N\to\infty}\sup_{x\in\hypercube}\abs{N \br{\exchange\br{x}-x}-\drift(x)} = 0.
             \]
        \end{hypothesis}
        This can be interpreted as follows: for each $i\in\range{1,m}$ and each $x_N\in\hypercube$, the number 
        \[
        N_i\br{\exchange^{(i)}\br{x_N}-x_N^{(i)}}
        \]
        is the net migration of the population $\alpha$ in the $i$-th patch if the densities of the population $\alpha$ on the patches before the migration are given by $x_N$. When the population size diverges to infinity, the net migration on patch $i$ will converge uniformly to $\distortion_i\drift[i]$. 
        
        Let $\drift\in\continuous[2]{\hypercube,\reals^m}$ be the function defined in \zcref{hyp:existence-b}. The assumptions made about the exchange functions imply certain properties of the function $\drift$. 

        For any $i\in\range{1,m}$, since the function $\exchange^{(i)}$ is $[0,1]$-valued, in particular,
        \[
            N\exchange^{(i)}(x) \geq 0 \text{ if } x_i=0 \text{ and } N\br{\exchange^{(i)}(x)-1}\leq 0 \text{ if } x_i=1.
        \]
        which leads to the following property:
        \begin{property}
            \zcsetup{reftype=property}
			\label{prop:b-boundary-non-negative}
			For any $x\in\hypercube$,
			\[
			\drift[i](x)\geq 0 \text{ if } x_i=0 \text{ and } \drift[i](x)\leq 0 \text{ if }x_i=1,
			\]
			or in other words, for any $x\in\partial\hypercube$,
			\[
				\scalarproduct{\drift(x),\normal(x)}\geq 0.
			\]
		\end{property}
        Given the conservation property defined in \zcref{hyp:conservation}, for any positive integers $N$ and $x$ in $\hypercube$,
        \[
            \scalarproduct{\distortion, N\br{\exchange\br{x}-x}} = 0,
       	\]
        thus:
        \begin{property}
            \zcsetup{reftype=property}
        	\label{prop:b-combi}
       		\[
       			\scalarproduct{\distortion,\drift} = \sum_{i=1}^m\distortion[i]\drift[i]=0.
       		\]
        \end{property}
        
        Finally, since the distortions are positive, in particular,  
        \[
            N\br{\exchange\br{\zero}-\zero} = N\br{\exchange\br{\one}-\one} = \zero.
        \]
        As a result:
        \begin{property}
            \zcsetup{reftype=property}
        	\label{prop:b-absorption}
        	\[
        		\drift(\zero) = \drift(\one) = 0.
        	\]
     	\end{property}

    \subsection{Main Results}
        Let $\diffusion$ be the diagonal matrix-valued function $\diffusion$ on $\hypercube$ defined by 
        \begin{equation}
            \label{def:a}
            \diffusion(x) = \diag\br{\diffusion[1](x),\dots, \diffusion[m](x)} = \diag\br{\frac{x_1(1-x_1)}{\distortion[1]},\dots, \frac{x_m(1-x_m)}{\distortion[m]}},
        \end{equation}
        for any $x\in\hypercube$, and define the second order differential operator
        \begin{equation}
            \label{def:L}
            \generator = \frac{1}{2}\sum_{i=1}^m \diffusion[i](x)\pd{!}{x_i,x_i}+\sum_{i=1}^m \drift[i](x)\pd{!}{x_i},~\domain{\generator}=\continuouscube[2],
        \end{equation}
        which is elliptic on $\hypercube$, but degenerates on the entire boundary.
        
        \begin{theorem}
            \zcsetup{reftype=theorem}
            \label{th:L-feller}
            Assume that the exchange functions $\exchange$ satisfy \zcref{hyp:conservation, hyp:existence-b} (hence, the function $\drift$ defined in \zcref{hyp:existence-b} satisfies \zcref{prop:b-boundary-non-negative}).
            Let $\generator$ be the operator defined by (\ref{def:L}), where $\diffusion$ is defined by (\ref{def:a}). Then the operator $\generator$ is closable, and its closure generates a Feller semigroup $\set{\semigroup(t)}[t\geq 0]$ on $\continuouscube$. In particular, $\continuouscube[2]$ is a core for the infinitesimal generator of this semigroup. 
        \end{theorem}
        
        Let $\Omega$ be the set $\continuous{\reals_+,\hypercube}$ endowed with the topology of uniform convergence on compact subsets of $\reals_+$, and $\scrB$ the associated Borel $\sigma$-algebra. Denote by $\process$ the canonical process on $\Omega$ defined by 
        \[
            \fun[\process_t]{\omega}{\process_t(\omega)=\omega(t)}[\Omega][\hypercube]
        \]
        for every $t\in\reals_+$, and by $\set{\scrB_t}$ its natural filtration. Once we have constructed the Feller semigroup $\set{\semigroup(t)}[t\geq 0]$, by Riesz Representation theorem and Daniell-Kolmogorov existence theorem, for any $x\in\hypercube$, there is a unique probability measure $\proba[x]{}^0$ such that the process $\process$ is a Feller process with initial distribution $\delta_x$ and whose semigroup is $\set{\semigroup(t)}[t\geq 0]$. 
        
        \begin{theorem}[Convergence in distribution]
            \zcsetup{reftype=theorem}
            \label{th:convergence}
            Let $\br{\process[N]}$ be the sequence of Feller processes on $\hypercube$ constructed in \zcref{subsec:finite-pop-model}, where the functions $\exchange$ satisfy \zcref{hyp:conservation,hyp:existence-b}, and $\drift$ the function in $\continuous[2]{\hypercube,\reals^m}$ given by \zcref{hyp:existence-b}. Let $\br{\mu_N}$ be the sequence of initial distributions of the processes $\process[N]$. If $\br{\mu_N}$ converges weakly to a probability measure $\mu$ on $\hypercube$, then $\br{\process[N]}$ converges to a limiting process with respect to the topology of \skd in the space $\rcll$ of càdlag functions $\reals_+\rightarrow\hypercube$, whose semigroup is $\set{\semigroup(t)}[t\geq 0]$ and whose initial distribution is $\mu$.
        \end{theorem}
        Now, let's focus on the problem of absorption of the diffusion.
        Due to \zcref{prop:b-absorption} and the fact that $\diffusion(\zero) = \diffusion(\one) = 0$, the states $\zero$ and $\one$ are absorbing. Since the diffusion coefficient $\diffusion$ degenerates on the entire boundary, if there exists a $1$-face, \ie a set of the form $\set{x\in\hypercube}[x_i=z]$, where $z\in\set{0,1}$, on which $\drift$ is null, then this face is absorbing. Therefore, the problem is controlled by the time it takes to reach this face or $\absorbingstates$, and it would reduce to a lower-dimensional problem  (as in \autocite{Ethier79Absorption}). On the other hand, if it happens that the drift pushes towards the interior of the hypercube everywhere on $\partial \hypercube\setminus\absorbingstates$, then the process is absorbed only if it specifically hits the corners $\zero$ and $\one$. This is the hardest situation (as raised in the introduction). The intermediate cases, where the drift $\drift$ is locally null on a subset of the boundary, would require careful analysis but do not raise any additional difficulties compared to the previous case. As a consequence, we will make the following additional assumption: 
        \begin{hypothesis}
            \zcsetup{reftype=hypothesis}
        	\label{hyp:b-boundary-positive}
     		For any $x\in \partial \hypercube\setminus\absorbingstates$,
			\[
				\scalarproduct{\drift(x),\normal(x)}>0.
			\]
        \end{hypothesis}

        \begin{notation}
            For any set $\calS$, we denote by $\hittingtime{\calS} = \inf\set{t\geq 0}[\process_t\in \calS] $ the hitting time of $\calS$ by the process $\process$. If $\calS$ is a singleton $\set{s}$ for some $s\in\hypercube$, we will denote it by $\hittingtime{s}$. 
        \end{notation}
        
        In order to study the extinction time of the diffusion process $\hittingtime{\absorbingstates}$, we will use the description of the process by the SDE corresponding to the differential operator $\generator$,
        \begin{equation}
        	\zcsetup{reftype=sde}
        	\tag{$E_{\dispersion,\drift}$}
            \label{sde:L}
            \d X_t = \dispersion\br{X_t}\d t + \drift\br{X_t}\d t,
        \end{equation}
        where $\dispersion$ is the symmetric positive square root of $\diffusion$ defined by
        \begin{equation}
            \label{def:sigma}
            \dispersion(x) =  \diag\br{\sqrt{\frac{x_1(1-x_1)}{\distortion[1]}},\dots, \sqrt{\frac{x_m(1-x_m)}{\distortion[m]}}}
        \end{equation}
        for any $x\in\hypercube$.
        
        Let $x\in \hypercube$. By Kolmogorov's equation, for any function $f\in\continuouscube[2]$, the process
        \[
            \set{f\br{\process_t}-\int_0^t\generator f\br{\process_s}\d s}[\scrB_t,~t\in\reals_+]
        \]
        is a $\proba[x]{}^0$-martingale, hence, the probability $\proba[x]{}^0$ solves the martingale problem for $\generator$ with initial distribution $\delta_x$. Then, by \autocite[Proposition 4.6 of Chapter 5 in][315]{KaratzasShreve1991}, there is a $m$-dimensional Brownian motion $\brownian = \set{\brownian_t}[\scrF_t;~t\in\reals_+]$ defined on an extension $\br{\Omega, \scrF,\proba[x]{}}$ of $\br{\Omega, \scrB,\proba[x]{}^0}$ such that $\br{\br{\process,\brownian},\br{\Omega,\scrF,\proba[x]{}}, \set{\scrF_t}}$ is a weak solution to \zcref{sde:L} with initial distribution $\delta_x$.
        \begin{theorem}[Absorption]
            \zcsetup{reftype=theorem}
            \label{th:hitting-time}
             Assume that the exchange functions $\exchange$ satisfy \zcref{hyp:conservation, hyp:existence-b} (hence, the function $\drift$ defined in \zcref{hyp:existence-b} satisfies \zcref{prop:b-combi, prop:b-absorption}), and furthermore that $\drift$ satisfies \zcref{hyp:b-boundary-positive}. 
            Then for any $x\in\hypercube$,
            \[
                \proba[x]{\hittingtime{\absorbingstates}<\infty}=1.
            \]
        \end{theorem}

        \begin{remark}
            The models and results can be generalised to cover any finite number of species, provided that the drift properties are adapted accordingly. A proof of the existence of a Feller semigroup would be based on this work. The fixation time, or the time at which all species except one become extinct, is also finite almost surely, and this can be proven essentially by using recurrence on the number of species and methods employed in \zcref{sec:extinction}. To avoid making this paper too lengthy, we have chosen to present the results for only two species.
        \end{remark}

\section{Construction of a Diffusion Feller Semigroup}

    The proof of \zcref{th:L-feller} developed in this section follows the method expanded in \autocite{Ethier1976}.
    We split the operator $\generator$ into two parts,
    \begin{equation}
        \label{def:A}
        \diffop = \frac{1}{2}\sum_{i=1}^m \frac{x_i(1-x_i)}{\distortion[i]}\pd{!}{x_i,x_i},~\domain{\diffop}=\continuouscube[2]
    \end{equation}
    and 
    \begin{equation}
        \label{def:B}
        \driftop = \sum_{i=1}^m \drift[i](x)\pd{!}{x_i},~\domain{\driftop}=\continuouscube[1].
    \end{equation}
    To apply the Trotter-Kato product formula result from \autocite[Exercise 9 of Chapter 1 in][44]{ethierkurtz:86:markov}, we must show that the closures of $\diffop$ and $\driftop$ both generate a strongly continuous semigroup, but we also need these semigroups to satisfy a good estimate on $\continuous[2]{}$. The following notations must be introduced for the upcoming constructions: 
    \begin{notation}\leavevmode
        \begin{itemize}
            \item For any multi-index $\alpha = (\alpha_1,\dots, \alpha_m)$, we denote by $\abs{\alpha}$ the sum $\alpha_1+\cdots+\alpha_m$, by $\partial^{\alpha}$ the partial derivative $\partial_1^{\alpha_1}\cdots \partial_m^{\alpha_m}$.
            \item For any positive integers $n$, we denote by $R_n$ the set of multi-indices $\set{\alpha\in\nintegers^m}[\abs{\alpha} \leq n]$ and by $R_n^*$ the set $R_n\setminus\set{\zero}$.
            \item We write $\borelfun{\hypercube}$ for the space of bounded Borel measurable functions $f:\hypercube\rightarrow\reals$, endowed with the sup-norm $\normC{f} = \sup_{x\in \hypercube}\abs{f(x)}$. For any $k\in\nintegers$, we write $\normC[k]{f} = \sum_{0\leq \abs{\alpha}\leq k}\normC{f}$ and $\seminormC{k}{f} = \sum_{1\leq \abs{\alpha}\leq k}\normC{f}$ for the respective norm and seminorm on $\continuouscube[k]$.
            \item For any positive integers $n$ and $k$, we write $\continuous[n,k]{\reals_+\times\hypercube}$ for the space of functions $\fun[f]{\reals_+\times\hypercube}{\reals}$ that are of class $\continuous[n]{}$ with respect to the first variable and of class $\continuous[k]{}$ with respect to the second variable.
        \end{itemize}
         
    \end{notation}
    \subsection{Construction of a Feller Semigroup from \texorpdfstring{$\diffop$}{A}}
        \begin{lemma}
            \zcsetup{reftype=lemma}
            \label{lem:A-estimate}
            Let $u\in\continuous[1,4]{\reals_+\times \hypercube}$ be a function satisfying
            \begin{equation}
        		\label{eq:A-evolution}
        		\partial_tu = \diffop u.
        	\end{equation}
            Then, for any $t\in\reals_+$,
            \begin{equation}
        		\zcsetup{reftype=inequality}
        		\label{ineq:A}
        		\seminormC{2}{u(t,\ \cdot )} \leq \seminormC{2}{u(0,\ \cdot )}.
        	\end{equation}
        \end{lemma}
        \begin{proof}
            Let $u\in\continuous[1,4]{\reals_+\times \hypercube}$ be a function satisfying (\ref{eq:A-evolution}) and $\alpha$ be a multi-index in $R_2^*$. Denoting by $u^{\alpha}$ the function $\partial^{\alpha}u$, we will show that there exists a dissipative operator $\diffop_{\alpha}$ with domain $\domain{\diffop}$ such that for any $t\in\reals_+$,
            \[
                u^{\alpha}(t,\ \cdot) = u^{\alpha}(0,\ \cdot) + \int_0^{t}\diffop_{\alpha}u^{\alpha}(s,\ \cdot)\d s.
            \]
            As a consequence, using \autocite[Proposition 2.10 of Chapter 1 in][15]{ethierkurtz:86:markov}, we will have
            \[
                \normC{u^{\alpha}(t,\ \cdot)}\leq \normC{u^{\alpha}(0,\ \cdot)},
            \]
            then, summing over $\alpha\in R_2^*$, we will obtain \zcref{ineq:A}. 
            
            Since the diffusion coefficients do not depend on time, by hypothesis,
            \[
                \partial_tu^{\alpha} = \partial^{\alpha}\diffop u.
            \]
            Now, 
            \[
                \partial^{\alpha}\diffop = \sum_{i=1}^m \sum_{0\leq \abs{\gamma}\leq \abs{\alpha}}\binom{\alpha}{\gamma} \partial^{\gamma}a_{i}(x)\partial^{\alpha-\gamma}\pd{!}{x_i,x_i} = \br{\diffop + \sum_{i=1}^m\frac{1}{2\distortion[i]}\br{\alpha_i\br{1-2x_i} \pd{!}{x_i} -\alpha_i\br{\alpha_i-1}}}\partial^{\alpha}.
            \]
            Thus, with $\diffop_{\alpha}=\diffop + \sum_{i=1}^m\frac{1}{2\distortion[i]}\br{\alpha_i\br{1-2x_i} \pd{!}{x_i} -\alpha_i\br{\alpha_i-1}}$, we obtain 
            \[
                \partial_tu^{\alpha} = \diffop_{\alpha}u^{\alpha}.
            \]
            It remains to prove that the operator $\diffop_{\alpha}$ is dissipative by verifying that it obeys the positive maximum principle. 
            
            Let $f\in\domain{\diffop}$ and $x\in\hypercube$ be such that $f(x) = \sup_{\hypercube}f\geq 0$. On one hand, for any $i\in\indices(x)$, $\frac{x_i(1-x_i)}{d_i}\pd{f}{x_i,x_i} =0$, $\pd{f}{x_i}(x)$ is non-positive if $i\in\indices[0](x)$ and non-negative if $i\in\indices[1](x)$ (refer to the beginning of \zcref{sec:model} for the notations). On the other hand, for any $j\in\range{1,m}\setminus\indices(x)$, $\pd{f}{x_j}(x) = 0$, and $\pd{f}{x_j,x_j}(x)\leq 0$. Hence, 
            \begin{align*}
                \diffop_{\alpha}f(x) 
                    & = \sum_{i=1}^m \br{\frac{x_i(1-x_i)}{2\distortion[i]} \pd{f}{x_i,x_i}+\frac{1}{2\distortion[i]}\br{\alpha_i\br{1-2x_i} \pd{f}{x_i} -\alpha_i\br{\alpha_i-1}}}\\ 
                    & = \sum_{i\in \indices[0](x)} \frac{\alpha_i}{2\distortion[i]}\underbrace{\pd{f}{x_i}(x)}_{\leq 0} -  \sum_{i\in \indices[1](x)} \frac{\alpha_i}{2\distortion[i]}\underbrace{\pd{f}{x_i}(x)}_{\geq 0} + \sum_{j\in \range{1,m}\setminus\indices(x)}\frac{x_j(1-x_j)}{2\distortion[j]}\underbrace{\pd{f}{x_j,x_j}(x)}_{\leq 0} -\sum_{i=1}^m\alpha_i\br{\alpha_i-1}\\
                    & \leq 0.
            \end{align*}
            As a result, $\diffop_{\alpha}$ obeys the positive maximum principle. The result follows.
        \end{proof}
        
        \begin{lemma}
            \zcsetup{reftype=lemma}
            \label{lem:A-feller}
            The operator $\diffop$ defined by (\ref{def:A}) is closable, and its closure generates a Feller semigroup $\set{\diffsemigroup(t)}[t\geq 0]$ on $\continuouscube$ that satisfies
            \begin{equation}
                \label{U-C2-invariance}
                \diffsemigroup(t):\continuouscube[2]\rightarrow\continuouscube[2]
            \end{equation}
            for any $t\in\reals_+$, and
            \begin{equation}
                \zcsetup{reftype=inequality}
                \label{ineq:U-C2-contraction}
                \seminormC{2}{\diffsemigroup(t)f}\leq \seminormC{2}{f}
            \end{equation}
            for any $t\in\reals_+$ and $f\in\continuouscube[2]$.
        \end{lemma}
        \begin{proof}\leavevmode
            \begin{step}
                \item We construct the semigroup.
        
                Let $x\in\hypercube$. The system of SDEs
                \[
                    Y_t^{(i)} = x_i + \frac{1}{\sqrt{\distortion[i]}}\int_0^t\sqrt{Y_s^{(i)}\br{1-Y_s^{(i)}}}\d B_s^{(i)},~i\in\range{1,m}
                \]
                has a weak solution $\br{\tilde{\process},\brownian, \tilde{\bfQ}_x}$, since every independent SDE has a weak solution (see \autocite[Theorem 6.1.1 of Chapter 6 in][184]{alfonsi2015affine}).
                Consequently, the measure $\bfQ_x = \tilde{\bfQ}_x\tilde{\process}^{-1}$ solves the martingale problem for $\diffop$ with initial distribution $\delta_x$. Define a family $\set{\diffsemigroup(t)}[t\geq 0]$ of linear operators $\continuouscube\rightarrow \boundedcube$ by 
                \begin{equation}
                    \label{def:U}
                    \diffsemigroup(t)f(x) = \expect[\bfQ_x]{f\br{\process_t}}
                \end{equation}
                for any $f\in\continuouscube$ and $x\in\hypercube$.
    
                \item We prove that $\set{\diffsemigroup(t)}[t\geq 0]$ induces a strongly continuous semigroup on the space $\polynomials$ of polynomial functions on $\hypercube$. 

                As the measure $\bfQ_x$ solves the martingale problem for $\diffop$, for any $t\in\reals_+$ and any function $f\in\continuouscube[2]$,
                \begin{equation}
                    \label{eq:A-generator}
                    \diffsemigroup(t)f(x) = f(x) + \int_0^t \diffsemigroup(s)\diffop f(x)\d s = f(x) + \int_0^t \diffop\diffsemigroup(s) f(x)\d s.
                \end{equation}
                Let $\polynomials[n]$ be the set of polynomial functions on $\hypercube$ of degree at most $n$. We observe that $\diffop$ induces a bounded linear operator $\diffop_n$ on $\polynomials[n]$. Let $\set{\diffsemigroup_n(t)}[t\geq 0]$ be the semigroup generated by $\diffop_n$ on $\polynomials[n]$, which is strongly continuous.  For any polynomial function $f\in\polynomials[n]$ and any $t\in\reals_+$, 
                \begin{multline*}
                    \normC{\diffsemigroup(t)f-\diffsemigroup_n(t)f} \leq \int_0^t\normC{\diffop\diffsemigroup(s)f-\diffop_n \diffsemigroup_n(s)f}\d s = \int_0^t\normC{\diffop_n\br{\diffsemigroup(s)f-\diffsemigroup_n(s)f}}\d s\\ 
                    \leq \norm[\polynomials[n]]{\diffop_n}\int_0^t \normC{\diffsemigroup(s)f-\diffsemigroup_n(s)f}\d s,
                \end{multline*}
                hence, by Grönwall's inequality, $\diffsemigroup(t)f = \diffsemigroup_n(t)f$. 

                As a result, for any $s,t\in\reals_+$,
                \begin{equation}
                    \label{U-polynomials-stable}
                    \diffsemigroup(t):\polynomials \rightarrow\polynomials
                \end{equation}
                and
                \begin{equation}
                    \label{U-polynomials-semigroup}
                    \diffsemigroup(s+t) = \diffsemigroup(s)\diffsemigroup(t) \text{ on }\polynomials,
                \end{equation}
                and for any $f\in\polynomials$,
                \begin{align}
                    \label{U-polynomials-strong-continuous}
                    \lim_{t\downarrow 0}\normC{\diffsemigroup(t) f-f} = 0.
                \end{align}
        
                \item We establish that $\set{\diffsemigroup(t)}[t\geq 0]$ is a Feller semigroup on $\continuouscube$.
                
                Let $f\in\continuouscube$. By the definition of the expected value, for each $t\in\reals_+$, the operator $\diffsemigroup(t)$ is non-negative and
                \begin{equation}
                	\zcsetup{reftype=inequality}
                    \label{ineq:U-contraction}
                    \normC{\diffsemigroup(t)f}\leq \normC{f}.\\
                \end{equation}
                By the density of $\polynomials$ in $\continuouscube$, there exists a sequence of polynomials $\br{f_n}$ such that 
                \[
                    \lim_{n\to\infty}\normC{f_n-f} = 0.
                \]
                Yet, for each $t\in\reals_+$ and $n\in\nintegers$, by (\ref{U-polynomials-stable}), $\diffsemigroup(t)f_n\in\polynomials$, and
                \[
                    \normC{\diffsemigroup(t) f_n-\diffsemigroup(t) f} = \sup_{x\in \hypercube} \abs{\expect[\bfQ_x]{f_n\br{\process_t} - f\br{\process_t}}}\leq \normC{f_n-f}.
                \]
                Hence, 
                \begin{equation}
                    \label{eq:limit-semigroup-polynomials}
                    \lim_{n\to\infty} \normC{\diffsemigroup(t)f_n-\diffsemigroup(t)f} = 0,
                \end{equation}
                and as a result, $\diffsemigroup(t)f\in\continuouscube$. Using (\ref{eq:limit-semigroup-polynomials}), (\ref{U-polynomials-semigroup}) and $(\ref{U-polynomials-strong-continuous})$, we conclude that $\set{\diffsemigroup(t)}[t\geq 0]$ is a Feller semigroup on $\continuouscube$. 
        
                \item We verify that $\diffsemigroup(t):\continuouscube[2]\rightarrow\continuouscube[2]$ for any $t\in\reals_+$.
        
                Let $f\in\continuouscube[2]$. By the density of $\polynomials$ in $\continuouscube[2]$, there exists a sequence of polynomials $\br{f_n}$ such that 
                \[
                    \lim_{n\to\infty}\normC[2]{f_n-f} = 0.
                \]
                For any non-negative integers $n$, using (\ref{U-polynomials-stable}) and \zcref{lem:A-estimate} with $u(t,x) = \diffsemigroup(t)f_n(x)$, we have 
                \begin{equation}
                    \zcsetup{reftype=inequality}
                    \label{ineq:U-estimate-polynomials}
                    \seminormC{2}{\diffsemigroup(t)f_n}\leq \seminormC{2}{f_n}.
                \end{equation}
                Combined with \zcref{ineq:U-contraction}, it turns out that $\br{\diffsemigroup(t)f_n}$ is a Cauchy sequence in the Banach space $\br{\continuouscube[2],\normC[2]{}}$. Consequently, the sequence converges in $\continuouscube[2]$. Due to the uniqueness of the limit, $\diffsemigroup(t)f\in \continuouscube[2]$. Taking the limit in \zcref{ineq:U-estimate-polynomials},  we obtain \zcref{ineq:U-C2-contraction}.   
            \end{step}
            We have constructed a Feller semigroup $\set{\diffsemigroup(t)}[t\geq 0]$ that satisfies (\ref{U-C2-invariance}). By \autocite[Proposition 17.9 of Chapter 17 in][374]{kallenberg2002foundations}, the subset $\continuouscube[2]$ of $\continuouscube$ is a core for the infinitesimal generator of the semigroup $\set{\diffsemigroup(t)}[t\geq 0]$. Furthermore, according to (\ref{eq:A-generator}), the generator coincides with $\diffop$ on $\continuouscube[2]$, so $\diffop$ is closable and its closure is the infinitesimal generator of $\set{\diffsemigroup(t)}[t\geq 0]$.
        \end{proof}

    \subsection{Construction of a Feller Semigroup from \texorpdfstring{$\driftop$}{B}}

        \begin{lemma}
            \zcsetup{reftype=lemma}
            \label{lem:B-estimate}
           Let $v\in\continuous[1,3]{\reals_+\times \hypercube}$ be a function satisfying
            \begin{equation}
        		\label{eq:B-evolution}
        		\partial_tv = \driftop v.
        	\end{equation}
            Then, there is a non-negative number $\lambda$ such that for any $t\in\reals_+$,
            \begin{equation}
        		\zcsetup{reftype=inequality}
        		\label{ineq:B}
        		\seminormC{2}{v(t,\ \cdot )} \leq e^{\lambda t}\seminormC{2}{v(0,\ \cdot )}.
        	\end{equation}
        \end{lemma}
        \begin{proof}
            Let $v\in\continuous[1,3]{\reals_+\times \hypercube}$ be a function satisfying (\ref{eq:B-evolution}). Denoting by $v^{\alpha}$ the function $\partial^{\alpha}v$ for any $\alpha\in R_2^*$, we will show that there are a dissipative operator $\bfD$ on the Banach space $\br{E={\continuouscube}^{R_2^*},\norm[E]{}}$, where 
            \[
                \norm[E]{f} = \sum_{\alpha\in R_2^*}\normC{f},
            \]
            and a non-negative number $\lambda$ such that the function $\fun[u]{t}{\br{e^{-\lambda t}v^{\alpha}(t,\ \cdot)}_{\alpha\in R_2^*}}$ satisfies
            \[
                u(t) = u(0) + \int_0^t \bfD u(s)\d s, 
            \]
            for any $t\in\reals_+$. Then, using again \autocite[Proposition 2.10 of Chapter 1 in][15]{ethierkurtz:86:markov}, we will obtain 
            \[
                \norm[E]{u(t)}\leq \norm[E]{u(0)},
            \]
            \ie
            \[
                \seminormC{2}{v(t,\ \cdot)}\leq e^{\lambda t}\seminormC{2}{v(0,\ \cdot)}.
            \]
            
            Since the coefficients do not depend on time, by hypothesis, for any $\alpha \in R_2^*$,
            \[
                \partial_tv^{\alpha} = \partial^{\alpha}\driftop u.
            \]
            Yet, we can write 
            \[
                \partial^{\alpha} \driftop = \br{\driftop+\sum_{\abs{\gamma} = 1}^{\abs{\alpha}}
                c_{\alpha,\gamma}v^{\gamma}}\partial^{\alpha},
            \]
            where $c_{\alpha,\gamma}$ are some functions in $\continuouscube$. Consequently, 
            \begin{equation}
                \label{eq:v-alpha}
                \partial_t v^{\alpha} = \driftop v^{\alpha} + \sum_{ \abs{\gamma} = 1}^{\abs{\alpha}}c_{\alpha,\gamma}v^{\gamma}.
            \end{equation}
            Let $\mathbf{\driftop}:\domain{\mathbf{\driftop}}\rightarrow E$ and $\mathbf{\operator}:E\rightarrow E$ be two linear operators on $E$ defined by 
            \[
                \mathbf{\driftop}(f) = \br{\driftop f_{\alpha}}_{\alpha\in R_2^*},~\domain{\mathbf{\driftop}} = \continuous[1]{\hypercube}^{R_2^*} \text{ and }\mathbf{\operator}(f) = \br{\sum_{\abs{\gamma} = 1}^{\abs{\alpha}} c_{\alpha,\gamma}f_{\alpha}}_{\alpha\in R_2^*}.
            \]
            As a consequence of \autocite[Proposition 3.23 of Chapter 2 in][88]{engel.nagel:00:one-parameter}, the sum of a dissipative operator and an operator that generates a strongly continuous contraction semigroup is dissipative on the intersection of their respective domains. Since the operator $\mathbf{\operator}$ is bounded, it generates a strongly continuous semigroup. To make its semigroup a contraction, it must be multiplied by $e^{-\lambda t}$, where $\lambda$ is a number larger than the norm of the operator $\mathbf{\operator}$. Therefore, in order to apply this result, we must find such a number $\lambda$ and prove that $\mathbf{\driftop}$ is dissipative.
            
            For any $f\in E$,
            \begin{align*}
        		\norm[E]{\mathbf{\operator} f} 
                    & = \sum_{\abs{\alpha} = 1}^2 
                            \normC{
                                \sum_{\abs{\gamma} = 1}^{\abs{\alpha}}
                                    c_{\alpha,\gamma}f_{\gamma}
                            }\\
        		      & \leq \sum_{\abs{\alpha} = 1}^2
                                \sum_{\abs{\gamma} = 1}^{\abs{\alpha}} 
                                    \normC{c_{\alpha,\gamma}}
                                    \normC{f_{\gamma}}\\ 
        		      & = \sum_{\abs{\gamma} = 1}^2
                            \sum_{\abs{\alpha} = \abs{\gamma}}^2
                                \normC{c_{\alpha,\gamma}}
                                \normC{f_{\gamma}}\\ 
        		      & = \br{
                            \max_{1\leq \abs{\gamma'}\leq 2}
                                \br{
                                    \sum_{\abs{\alpha} = \abs{\gamma'}}^2
                                        \normC{c_{\alpha,\gamma'}}
                                }     
                        }
                            \sum_{\abs{\gamma}=1}^2\normC{f_{\gamma}}\\ 
            		  & = \br{
                                \max_{1\leq \abs{\gamma'}\leq 2}
                                    \br{
                                        \sum_{\abs{\alpha} = \abs{\gamma'}}^2
                                            \normC{c_{\alpha,\gamma'}}
                                    }     
                            }
                                \norm[E]{f}.\\ 
        	\end{align*}
            Set $\lambda = 
                \max_{1\leq \abs{\gamma'}\leq 2}
                \br{
                    \sum_{\abs{\alpha} = \abs{\gamma'}}^2
                        \normC{c_{\alpha,\gamma'}}
                }$.
            Then, the operator $\mathbf{C}-\lambda$ generates a strongly continuous contraction semigroup. 
            
            It remains to prove that $\mathbf{\driftop}$ is dissipative. First, let's demonstrate that the operator $\driftop$ is dissipative on $\domain{\driftop}= \continuouscube[1]$ by verifying that it obeys the positive maximum principle. 
    
            Let $f\in \domain{\driftop}$ and $x\in \hypercube$ such that $f(x) = \sup_{\hypercube}f\geq 0$. On one hand, for any $i\in\indices[0](x)$, $\drift[i](x)\geq 0$ and $\pd{f}{x_i}(x)\leq 0$, so the product $\drift_i(x)\pd{f}{x_i}(x)$ is non-positive, and for any $i\in\indices[1](x)$, $\drift[i](x)\leq 0$ and $\pd{f}{x_i}(x)\geq 0$, so the product $\drift[i](x)\pd{f}{x_i}(x)\leq 0$ is also non-positive. On the other hand, for any $j\in\range{1,m}\setminus\indices(x)$, $\pd{f}{x_j} = 0$. Hence, summing over $i\in\indices(x)$ and $j\in\range{1,m}\setminus\indices(x)$, we obtain
            \[
                \driftop f(x)\leq 0.
            \]
            
            Now, let's prove that the operator $\mathbf{\driftop}$ is dissipative. For any function $f\in \domain{\mathbf{\driftop}}$ and real $\gamma >0$, by the dissipativity of the operator $\driftop$,
            \[
                \norm[E]{\gamma f-\mathbf{\driftop}f} = \sum_{1\leq\abs{\alpha}\leq 2}\normC{\gamma f_{\alpha} - \driftop f_{\alpha}}\geq \sum_{1\leq\abs{\alpha}\leq 2}\gamma\normC{f_{\alpha}} = \gamma\norm[E]{f}.
            \]
            All things considered, the operator $\mathbf{\driftop}+\mathbf{\operator}-\lambda$ is dissipative, and the function $\fun[u]{t}{\br{e^{-\lambda t}v^{\alpha}(t,\ \cdot)}_{\alpha\in R_2^*}}$ satisfies
            \[
                u(t) = u(0) + \int_0^t \br{\mathbf{\driftop}+\mathbf{\operator}-\lambda}u(s)\d s
            \]
            for any $t\in\reals_+$, and the result follows.
        \end{proof}
        
        \begin{lemma}
            \zcsetup{reftype=lemma}
            \label{lem:B-feller}
            The operator $\driftop$ defined by (\ref{def:B}) is closable, and its closure generates a Feller semigroup $\set{\driftsemigroup(t)}[t\geq 0]$ on $\continuouscube$ that satisfies
            \begin{equation}
                \label{V-C2}
                \driftsemigroup(t):\continuouscube[2]\rightarrow\continuouscube[2]
            \end{equation}
            for any $t\in\reals_+$, and
            \begin{equation}
                \zcsetup{reftype=inequality}
                \label{ineq:V-C2-contraction}
                \seminormC{2}{\driftsemigroup(t)f}\leq \seminormC{2}{f}
            \end{equation}
            for any $t\in\reals_+$ and $f\in\continuouscube[2]$.
        \end{lemma}
        
        \begin{proof}
            Consider the function 
            \[
                \fun[\rho]{y}{\argmin_{z\in \hypercube}\abs{y-z}}[\reals^m][\hypercube].
            \]
            Since the function $\drift$ is in $\continuous[2]{\hypercube,\reals^m}$, $\drift \circ \rho$ is Lipschitz continuous on $\reals^m$. By Picard-Lindelöf theorem, for any $x\in \hypercube$, there is a unique solution $y:\reals_+\times \reals^m\rightarrow\reals^m$ to the initial value problem
            \[
                \begin{cases}
                    \dot{u} = \drift\circ\rho(u) &\text{in } \reals_+^*\times\reals^m\\ 
                    u(0,x) = x & \text{for any }x\in \reals^m.
                \end{cases}
            \]
            Define the sets 
            \[
                H_i(0) = \set{x\in\reals^m}[x_i < 0] \text{ and } H_i(1) = \set{x\in\reals^m}[x_i > 1],
            \]
            for each $i\in\range{1,m}$. Fix $x\in\hypercube$ and let $y_x= y(\cdot, x)=\br{y_x^{(1)},\dots, y_x^{(m)}}$, where $y$ is the unique solution to the initial value problem. Suppose that there exists $t>0$ such that $y_x(t)\in\reals^m\setminus \hypercube$. Then, there exist $z\in\set{0,1}$ and $i\in\range{1,m}$ such that $y_x(t)\in H_i(z)$. Denote by $\face = \set{y\in\hypercube}[y_i = z]$ and $\normal[\face]$ an inward direction to the face $\face$. Due to the continuity of $y_x$, there exist $0<t_1<t_2$ such that $\scalarproduct{y_x(t_2)-y_x(t_1),\normal[\face]}<0$, and for any $s\in[t_1,t_2]$, $y_x(s)\in H_i(z)$. Nevertheless, we also have 
            \[
                \scalarproduct{y_x(t_2)-y_x(t_1),\normal[\face]} = \int_{t_1}^{t_2}\scalarproduct{\drift\circ\rho(y_x(s)),\normal[\face]}\d s.
            \]
           Provided that $\rho\br{H_i(z)}\subset \face$, for any $s\in [t_1,t_2]$, $\scalarproduct{\drift\circ\rho(y_x(s)),\normal[\face]}\geq 0$, and this leads to a contradiction. Thus, for any $t\in\reals_+$, $y_x(t)\in \hypercube$, and 
            \[
                y_x(t) = x + \int_0^t\drift(y_x(s))\d s.
            \]
            And since the function $\drift$ is in $\continuous[2]{\hypercube,\reals^m}$, then $y\in\continuous[1,2]{\reals_+\times \hypercube,\hypercube}$. Define the family of linear operators $\set{\driftsemigroup(t)}[t\geq 0]$ on $\continuouscube$ by 
            \[
                \driftsemigroup(t)f(x) = f(y(t,x)),~t\in\reals_+,~f\in\continuouscube.
            \]
            By construction, $\set{\driftsemigroup(t)}[t\geq 0]$ is a Feller semigroup on $\continuouscube$ that satisfies (\ref{V-C2}). In particular, $\continuouscube[2]$ is a core for the infinitesimal generator of $\set{\driftsemigroup(t)}[t\geq 0]$. Furthermore,
            \[
                \partial_t\driftsemigroup(t)f = \driftsemigroup(t)\driftop f,~t\in\reals_+,~f\in\continuouscube[1],
            \]
            so its infinitesimal generator coincides with $\driftop$ on $\continuouscube[1]$.
            Finally, for any $f\in\continuouscube[2]$, using \zcref{lem:B-estimate} with $v(t,x) = \driftsemigroup(t)f(x)$, we obtain \zcref{ineq:V-C2-contraction}.
        \end{proof}

    \subsection{Construction of the Feller Semigroup from \texorpdfstring{$\generator$}{L}}
        We are now able to prove \zcref{th:L-feller}.
        \begin{proof}[Proof of \texorpdfstring{\zcref{th:L-feller}}]
            Consider the dense subspace $D = \continuouscube[2]$ of $\continuouscube$, equipped with the norm $\normC[2]{}$, so that $D$ is a Banach space. We have
            \[
                \normC[2]{f}\geq \normC{f}
            \]
            for any $f\in D$. Let $\set{\diffsemigroup(t)}[t\geq 0]$ and $\set{\driftsemigroup(t)}[t\geq 0]$ be the semigroups obtained in \zcref{lem:A-feller, lem:B-feller}, which are generated by $\closure{\diffop}$ and $\closure{\driftop}$ respectively, and satisfy
            \[
                \diffsemigroup(t): D\rightarrow D \text{ and }\driftsemigroup(t): D\rightarrow D
            \]
            for any $t\in\reals_+$. Insofar as $D$ is a subset of the domains of both operators $\diffop$ and $\driftop$, $\closure{\diffop}$ coincides with $\diffop$ on $D$ and likewise for $\closure{\driftop}$ and $\driftop$.
            
            Since the functions $\drift[i]$ are in $\continuouscube[2]$, $D$ is a subset of $\domain{\driftop^2}$. Therefore, there exists a non-negative number $\gamma$ such that for any function $f\in D$,
            \[
                \normC{\driftop^2f}\leq \gamma\normC[2]{f}.
            \]
            Set $\mu = \max(\gamma,\lambda)$, where $\lambda$ is as in \zcref{lem:B-feller}. Then, for any $t\in\reals_+$ and $f\in D$,
            \[
                \normC[2]{\diffsemigroup(t)f}\leq e^{\mu t}\normC[2]{f} \text{ and } \normC[2]{\driftsemigroup(t)f}\leq e^{\mu t}\normC[2]{f}.
            \]
            By applying the result of \autocite[Exercise 9 of Chapter 1 in][44]{ethierkurtz:86:markov}, the closure of the restriction of $\closure{\diffop}+\closure{\driftop}$ generates a Feller semigroup. Yet, the restriction of $\closure{\diffop}+\closure{\driftop}$ to $D$ is $\generator$. Hence, $\closure{\generator}$ generates a Feller semigroup $\set{\semigroup(t)}[t\geq 0]$.
        \end{proof}

\section{Convergence in Distribution}
\zcsetup{reftype=section}
\label{sec:convergence}
    In the previous section, we proved that the closure of the operator $\br{\generator,\continuouscube[2]}$ generates a Feller semigroup $\set{\semigroup(t)}[t\geq 0]$, which means that $\continuouscube[2]$ is a core for its infinitesimal generator. This allows us to establish the convergence in distribution.
    
    \begin{proof}[Proof of \texorpdfstring{\zcref{th:convergence}}]
        Set $D= \continuouscube[2]$. By \zcref{th:L-feller}, the subset $D$ of $\continuouscube$ is a core for the infinitesimal generator of the semigroup $\set{\semigroup(t)}[t\geq 0]$. Using \autocite[Theorem 17.25 of Chapter 17 in][386]{kallenberg2002foundations}, it suffices to prove that for any $f\in D$, 
        \[
            \lim_{N\to\infty}\normC{\generator[N]f-\generator f} = 0.
        \]
         Fix a function $f\in D$. The $m$-dimensional Bernstein operator $\bernstein[N]$ is equal to the product $\bernstein[N_1]^{(1)}\cdots \bernstein[N_m]^{(m)}$, where $\bernstein[N_i]^{(i)}$ is the one dimensional Bernstein operator acting on the $i$-th coordinate,
        \[
            B_N^{(i)}f(x) = \sum_{k=0}^{N_i}\binom{N_i}{k}x_i^{k}\br{1-x_i}^{N_i-k}f\br{x_1,\dots,x_{i-1} \frac{k}{N_i}, x_{i+1},\dots, x_m }.
        \]
        Using the fact that for all $i\in\range{1,m}$,
        \[
            \lim_{N\to\infty} N\br{\bernstein[N_i]^{(i)}f-f}=\frac{x_i(1-x_i)}{2d_i}\pd{f}{x_i,x_i}
        \]
        (see, \eg \autocite[Theorem 3.9 in][206]{bustamante}), we obtain, by a successive application of the operators $\bernstein[N_i]^{(i)}$, that
        \[
            \lim_{N\to\infty} \normC{N\br{B_Nf-f} -\diffop} = 0.
        \]
        By Taylor's inequality, for any integers $N$ and $x\in\hypercube$,
        \begin{align*}
            \abs{f\circ\exchange(x)-f(x) -\scalarproduct{\exchange(x)-x,\nabla f(x)}}\leq \sup_{\hypercube}\norm{H_f}\abs{\exchange(x)-x}^2,
        \end{align*}
        where $H_f$ is the Hessian matrix of $f$. It follows that 
        \begin{align*}
            \abs{N\br{f\circ\exchange(x) -f(x)}-\scalarproduct{b(x),\nabla f(x)} } 
                & \leq \abs{N\br{f\circ\exchange(x) -f(x) - \scalarproduct{\exchange(x)-x, \nabla f(x)}}} \\
                & \phantom{===} + \abs{\scalarproduct{N\br{\exchange(x)-x}, \nabla f(x) }- \scalarproduct{b(x),\nabla f(x)}}\\
                & \leq \sup_{\hypercube}\norm{H_f}N\abs{\exchange\br{x}-x}^2\\ 
                & \phantom{===} + \sup_{\hypercube}\abs{\nabla f}\abs{N\br{\exchange(x)-x}-b(x)}\\ 
                & \leq  \sup_{\hypercube}\norm{H_f}N\sup_{x\in\hypercube}\abs{\br{\exchange(x)-x}}^2 \\
                & \phantom{===} + \sup_{\hypercube}\abs{\nabla f} \sup_{x\in\hypercube}\abs{N\br{\exchange(x)-x}-\drift(x)}.
        \end{align*}
        Thus, by \zcref{hyp:existence-b},
        \[
            \lim_{N\to\infty}\normC{N\br{f\circ\exchange -f}-\driftop f} = 0.
        \]
       Now, by the linearity of the Bernstein operator, 
       \begin{align*}
            \normC{N\br{\transition f-f} -\generator f} 
                & = \normC{ N\br{\bernstein\br{f\circ\exchange}-f}-\generator f}\\ 
                & \leq \normC{\br{\bernstein\br{N\br{f\circ \Phi_N-f} -\driftop f}}} + \normC{\bernstein\br{\driftop f}-\driftop f}\\
                 & \phantom{===} + \normC{N\br{\bernstein f -f}-\diffop f}.
        \end{align*}
        Hence,
        \[
            \lim_{N\to\infty} \normC{\generator[N]f-\generator f} = 0.
        \]
    \end{proof}

\section{Extinction Time}
    \zcsetup{reftype=section}
    \label{sec:extinction}
    Recall that for each $x\in\hypercube$, $\proba[x]{}^0$ is the probability under which the canonical process $\process$ defined on $\Omega=\continuous{\reals_+,\hypercube}$ with its Borel $\sigma$-algebra $\scrB$ is a Feller process with semigroup $\set{\semigroup(t)}[t\geq 0]$ and initial distribution $\delta_x$. Recall also that there is a $m$-dimensional Brownian motion $\brownian = \set{\brownian_t}[\scrF_t;~t\in\reals_+]$ defined on an extension $\br{\Omega, \scrF,\proba[x]{}}$ of $\br{\Omega, \scrB,\proba[x]{}^0}$ such that the triplet $\br{\br{\process,\brownian},\br{\Omega,\scrF,\proba[x]{}}, \set{\scrF_t}}$ is a weak solution to \zcref{sde:L} with initial distribution $\delta_x$. The objective of this section is to provide a proof of \zcref{th:hitting-time}. Recall that $\absorbingstates$ is the set of absorbing states $\set{\zero,\one}$, and that the hitting time of $\absorbingstates$ by the process $\process$ is denoted by $\hittingtime{\absorbingstates}$. We aim to demonstrate that the function $\fun{x}{\proba[x]{\hittingtime{\absorbingstates}<\infty}}$ is constant and is equal to $1$.

    Let's define for any $\alpha\in (0,\min\br{1,\distortioncombi(\one)})$, the sets
    \begin{equation}
        \label{def:Delta}
        \absorbzone^0  = \set{x\in\hypercube}[\distortioncombi(x)<\alpha] \text{ and }\absorbzone^1  = \set{x\in\hypercube}[\distortioncombi(x)>\alpha],
    \end{equation}
    and write $\absorbzone$ for the reunion $\absorbzone^0\cup\absorbzone^1$ (see \zcref{fig:Delta}).
        \begin{figure}[ht]
        \centering
        \begin{tikzpicture}[scale=0.6]

            \coordinate (A) at (0,0);
            \coordinate (B) at (10,0);
            \coordinate (C) at (10,10);
            \coordinate (D) at (0,10);

            \fill[fill=colorDelta] (0,0)--(5,0)--(0,4) node[color=colorDelta!50!black,xshift=8mm,yshift=-16mm] {$\absorbzone^0$};
            \fill[fill=colorDelta] (10,10)--(5,10)--(10,6)node[color=colorDelta!50!black,xshift=-8mm,yshift=16mm] {$\absorbzone^1$};
            % côtés
            \draw[line width=1pt] (A) -- (B);
            \draw[line width=1pt] (B) -- (C);
            \draw[line width=1pt] (C) -- (D);
            \draw[line width=1pt] (D) -- (A);
        
            \draw[line width=1.5pt, color=colorDeltaBound] (0,0) -- (5,0) node[midway, left] {};
            \draw[line width=1.5pt, color=colorDeltaBound] (0,0) -- (0,4) node[midway, left] {};
        
            \draw[line width=1.5pt, color=colorDeltaBound] (10,10) -- (5,10) node[midway, left] {};
            \draw[line width=1.5pt, color=colorDeltaBound] (10,10) -- (10,6) node[midway, left] {};
            
            \filldraw[color=colorAbsorbing] (A) circle (4pt) node[anchor=north east, colorAbsorbing]{$(0,0)$};
            \filldraw[color=colorAbsorbing] (C) circle (4pt) node[anchor=south west, colorAbsorbing]{$(1,1)$};
        \end{tikzpicture}
        \caption{An Example of $\absorbzone^0$ and $\absorbzone^1$ ($m=2$)}
        \label{fig:Delta}
    \end{figure}
    Introduce the functions 
    \begin{equation}
        \label{def:halpha}
        \fun[h_{\alpha}]{x}{\proba[x]{\hittingtime{\absorbzone}<\infty}}[\hypercube][[0,1]]
    \end{equation}
    and
    \begin{equation}
        \label{def:h}
        \fun[h]{x}{\proba[x]{\hittingtime{\absorbingstates}<\infty}}[\hypercube][[0,1]].
    \end{equation}
    We split the proof into two parts. First, we show that starting from any state $x$ of $\absorbzone$, the process will be absorbed within a finite time with a positive probability. In other words, we prove that for any $x\in\absorbzone$, we have $h(x)>0$. Secondly, we prove that starting from any point in $\hypercube$, the process will reach $\absorbzone$ within a finite time with a positive probability, \ie $h_{\alpha}$ is positive. Then, using the fact that this function is lower semicontinuous, we show that its minimum equals $1$. Finally, we use the Markov property to show that the function $h$ is constant and equal to $1$ in $\hypercube$. 
    
    \subsection{Attainability of a Neighbourhood of the Absorbing States}
        The idea behind proving that starting from $\absorbzone$, the process is absorbed within a finite time is to exploit \zcref{prop:b-combi} to reduce the problem to the study of an Itô process in one dimension.
        
        Denote by $\distortionprod$ the constant $\distortion[1]\times \cdots \times \distortion[m]$ and introduce the function 
        \[
            \fun[\distortioncombi]{y}{\frac{1}{\distortionprod}\scalarproduct{\distortion,y}}[\hypercube][\reals_+].
        \]        
        \begin{lemma}
            \zcsetup{reftype=lemma}
            \label{lem:Delta-absorption}
            Let $\alpha$ be a positive number such that $\alpha\leq \min\br{1,\distortioncombi(\one)}$. Let $\absorbzone^0$ and $\absorbzone^1$ be the sets defined in (\ref{def:Delta}).
            Then for any $x\in \absorbzone^0$, 
            \begin{equation}
            	\zcsetup{reftype=inequality}
                \label{ineq:Delta-absorption-0}
                \proba[x]{\hittingtime{\zero}<\infty}\geq 1-\frac{\distortioncombi(x)}{\alpha}>0,
            \end{equation}
            and for any $x\in \absorbzone^1$,
            \begin{equation}
            	\zcsetup{reftype=inequality}
                \label{ineq:Delta-absorption-1}
                \proba[x]{\hittingtime{\one}<\infty}\geq 1-\frac{\distortioncombi(\one-x)}{\alpha}>0.
            \end{equation}
        \end{lemma}

        \begin{proof}
              Fix $x\in\hypercube$. Under the probability $\proba[x]{}$, the process $\br{\process,\brownian}$ is a weak solution to the SDE 
             \begin{equation}
             	\zcsetup{reftype=sde}
                \label{sde:L-x}
                \tag{$E_{\dispersion,\drift}(x)$}
                 \process_t^{(i)} = x_i + \frac{1}{\sqrt{\distortion[i]}}\int_0^t \sqrt{\process_s^{(i)}\br{1-\process_s^{(i)}}}\d \brownian_s^{(i)} + \int_0^t \drift[i]\br{\process_s}\d s,~i\in\range{1,m}.
             \end{equation}
             Set $\processcombi = \distortioncombi(\process)$. For any subset $\calS$ of $\distortioncombi(\hypercube)$, denote by $\hittingtime{\calS}^\processcombi$ the hitting time of $\calS$ by the process $\processcombi$, and by $\hittingtime{s}^\processcombi$ if $\calS=\set{s}$ for some $s\in\distortioncombi(\hypercube)$. 
             
             Let us notice that since the distortion coefficients $\distortion[i]$ are positive, we have $\hittingtime{\zero}^{\processcombi} = \hittingtime{\zero}$ and $\hittingtime{\one}^{\processcombi} = \hittingtime{\one}$. By \zcref{sde:L-x} and \zcref{prop:b-combi}, for all $t\in\reals_+$, under $\proba[x]{}$,
            \[
                \processcombi_t = \distortioncombi(x)+\sum_{i=1}^m\int_0^t\sqrt{\frac{\distortion[i]}{\distortionprod}\frac{{\process_s}^{(i)}\br{1-{\process_s}^{(i)}}}{\distortionprod}}\d \brownian_s^{(i)}.
            \]
            Consider, for all $i\in\range{1,m}$, 
            \[
               \psi_i(x) = \frac{1}{m}\1_{\set{0,1}^m}(x)+ \sqrt{\frac{\distortion[i]x_i(1-x_i)}{\sum_{j=1}^m \distortion[j]x_j(1-x_j)}}\1_{K\setminus\set{0,1}^m}(x),
            \]
            and set, for all $t\in\reals_+$,
            \[
                \browniancombi_t = \sum_{i=1}^m\int_0^t\psi_i\br{\process_s}\d \brownian_s^{(i)},
            \]
            which is well defined because the functions $\psi_i$ are bounded and measurable. By construction, the process $\browniancombi$ is a local martingale. Since the Brownian motions $\brownian^{(i)}$ are independent, for any $t\in\reals_+$,
            \begin{multline*}
                \scalarproduct{\browniancombi}_t =  \sum_{i=1}^m\br{\int_0^t \frac{1}{m}\1_{\set{0,1}^m}(\process_s)\d s + \int_0^t 	\frac{d_i{\process_s}^{(i)}(1-{\process_s}^{(i)}) }{
            \sum_{j=1}^m d_j{\process_s}^{(j)}\br{1-{\process_s}^{(j)}}  }\1_{K\setminus\set{0,1}^m}({\process_s})\d s}\\ 
            = \int_0^t \1_{\set{0,1}^m}(\process_s) + \1_{K\setminus\set{0,1}^m}(\process_s)\d s = t.
            \end{multline*}
            Hence, by Levy's characterisation of Brownian motion, the process $\browniancombi$ is a Brownian motion. Observe that  
            \[
                \sum_{j=1}^m \frac{\distortion[j]}{\distortionprod}{\process_s}^{(j)} \br{1-{\process_s}^{(j)}} = \processcombi_s\br{1-\processcombi_s} + Z_s,
            \]
            where 
            \[
                Z_s = \sum_{i=1}^m\frac{d_i}{\distortionprod }\br{ \frac{\distortion[i]}{\distortionprod}-1}\br{{\process_s}^{(i)}}^2 +2\sum_{1\leq i<j\leq m} \frac{\distortion[i]\distortion[j]}{\distortionprod^2}{\process_s}^{(i)}{\process_s}^{(j)}.
            \]
            Moreover, since for any $i\in\range{1,m}$, we have $\distortion[i]\leq 1$, then $\absorbzone\cap\br{\set{0,1}^m\setminus \absorbingstates}=\varnothing$. Therefore, if $\processcombi_0$ is supported in $(0,\alpha)$, then for any $t\leq \hittingtime{\set{0,\alpha}}^{\processcombi}$, $\br{\processcombi_t,\browniancombi_t}$ satisfies
            \[
                \processcombi_t = \processcombi_0 + \frac{1}{\sqrt{\distortionprod}}\int_0^t\sqrt{\processcombi_s\br{1-\processcombi_s} + Z_s}\d \browniancombi_s,
            \]
            and $\processcombi$ is a martingale and almost a Wright-Fisher diffusion. Suppose that $x\in\absorbzone^0$. Consider the non-negative function
            \[
                \fun[u]{r}{-2\br{ r\ln(r) +(1-r)\ln(1-r)}}[[0,\alpha]][\reals_+].
            \]
            which is of class $\continuous[2]{}$ on $(0,\alpha)$ and we have, for any $r\in (0,\alpha)$,
            \[
                u''(r) = -\frac{2}{r(1-r)}.
            \]
            By Itô's rule, since the function $u''$ is negative on $(0,\alpha)$, the process $Z$ is non-negative, and the martingale part equals $0$ at $t=0$, then for any $t\in\reals_+$,
            \begin{align*}
                \expect[x]{u\br{\processcombi_{t\wedge \hittingtime{\set{0,\alpha}}^{\processcombi} }}} 
                    & = u(\distortioncombi(x)) + \expect[x]{\frac{1}{\distortionprod}\int_0^{t\wedge\hittingtime{\set{0,\alpha}}^{\processcombi}}u''(\processcombi_s)\processcombi_s\br{1-\processcombi_s} +u''(\processcombi_s)Z_s \d s} \\ 
                    & = u(\distortioncombi(x)) -\frac{2}{\distortionprod}\expect[x]{t\wedge\hittingtime{\set{0,\alpha}}^{\processcombi}} + \expect[x]{\frac{1}{\distortionprod}\int_0^{t\wedge \hittingtime{\set{0,\alpha}}^{\processcombi}}u''(\processcombi_s)Z_s\d s}\\
                    & \leq u(\distortioncombi(x))-\frac{2}{\distortionprod}\expect[x]{t\wedge \hittingtime{\set{0,\alpha}}^{\processcombi}}.
            \end{align*}
            Because the function $u$ is non-negative,
            \[
                \expect[x]{t\wedge \hittingtime{\set{0,\alpha}}^{\processcombi}}\leq u\br{\distortioncombi(x)}.
            \]
            By the monotone limit theorem, 
            \[
                \expect[x]{\hittingtime{\set{0,\alpha}}^{\processcombi}}<\infty,
            \]
            and $\proba[x]{\hittingtime{\set{0,\alpha}}^{\processcombi}<\infty}=1$. Hence,
            \[
                \proba[x]{\hittingtime{0}^{\processcombi} =\hittingtime{\alpha}^{\processcombi}=\infty} = 0,
            \]
            and
            \[
                \proba[x]{\hittingtime{0}^{\processcombi} = \infty} = \proba[x]{\hittingtime{0}^{\processcombi}=\infty,~\hittingtime{\alpha}^{\processcombi}<\infty}\leq \proba[x]{\hittingtime{\alpha}^{\processcombi}<\hittingtime{0}^{\processcombi}}.
            \]
            Since $\processcombi$ is a martingale and the stopping time $\hittingtime{\set{0,\alpha}}^{\processcombi}$ is integrable, by the optional stopping theorem, 
            \[
                \distortioncombi(x) = \expect[x]{\processcombi_{\hittingtime{\set{0,\alpha}}^{\processcombi}} } = \alpha\proba[x]{\hittingtime{\alpha}^{\processcombi}<\hittingtime{0}^{\processcombi}}.
            \]
            As a result, 
            \[
                \proba[x]{\hittingtime{\zero}<\infty}>1-\frac{\distortioncombi(x)}{\alpha}.
            \]
            Suppose that $x\in\absorbzone^1$. By noticing that $\absorbzone^1 = \distortioncombi\br{\one- \cdot }^{-1}\br{[0,\alpha)}$, we replace $\processcombi = \distortioncombi(\process)$ with $\distortioncombi(\one-\process)$, $\browniancombi$ with $-\browniancombi$ and $Z_s$ with 
            \[
                Z_s = \sum_{i=1}^m\frac{d_i}{\distortionprod }\br{ \frac{d_i}{\distortion}-1}\br{1-{\process_s}^{(i)}}^2 +2\sum_{1\leq i<j\leq m} \frac{d_id_j}{\distortionprod^2}\br{{1-\process_s}^{(i)}}\br{{1-\process_s}^{(j)}},
            \]
            to obtain, by the same arguments, that
            \[
                \proba[x]{\hittingtime{\one}<\infty}>1-\frac{\distortioncombi(\one-x)}{\alpha}.
            \]
        \end{proof}

    \subsection{Attainability of the Absorbing States}        
       In this subsection, we first prove that starting from any point in $\hypercube$, the process hits the subset $\absorbzone$ with positive probability. Then, we demonstrate that the process hits the absorbing states in finite time almost surely. 
        \begin{notation}
        	For any $t\in\reals_+$, we denote by $\theta_t$ the shift transformation
        	\[
        		\fun[\theta_t]{\omega}{\theta_t(\omega) = \set{\omega_{s+t}}[s\in\reals_+]}[\Omega].
        	\]
        \end{notation}
        \begin{lemma}
            Let $\alpha$ be a positive number such that $\alpha\leq \min\br{1,\distortioncombi(\one)}$. Let $h$ and $h_{\alpha}$ be the functions defined by (\ref{def:halpha}) and (\ref{def:h}) respectively. The processes $h_{\alpha}(\process)=\set{h_{\alpha}(\process_t)}[\scrF_t,~t\in\reals_+]$ and $h(\process) =\set{h(\process_t)}[\scrF_t,~t\in\reals_+]$ are, respectively, a supermartingale and a martingale.
        \end{lemma}
        \begin{proof}
            First, let us notice that since $\absorbingstates$ is absorbing. then for any $t\in\reals_+$,
            \[
                \1_{\set{\hittingtime{\absorbingstates}<\infty}}\circ \theta_t = \1_{\set{\hittingtime{\absorbingstates}<\infty}},
            \]
            and for any $s,t\in\reals_+$,
            \[
                \1_{\set{\hittingtime{\absorbzone}<\infty}}\circ \theta_s\circ\theta_t \leq \1_{\set{\hittingtime{\absorbzone}<\infty}}\circ \theta_s.
            \]
            By Markov property, for any $s,t\in\reals_+$,
            \begin{align*}
                \expect[x]{h_{\alpha}(\process_{s+t})}[\scrF_s] 
                & = \expect[x]{
                        \expect[\process_{s+t}]{
                            \1_{\set{\hittingtime{\absorbzone}<\infty}}
                            }     
                    }[\scrF_s]\\ 
                & = \expect[x]{
                        \expect[x]{
                            \1_{\set{\hittingtime{\absorbzone}<\infty} }\circ \theta_{s+t}
                        }[\scrF_{s+t}]
                    }[\scrF_s]\\ 
                & \leq \expect[x]{
                        \1_{\set{\hittingtime{\absorbzone}<\infty} }\circ \theta_{s}
                    }[\scrF_s]\\ 
                & = h_{\alpha}\br{X_s}.
            \end{align*}
            and likewise,
            \begin{align*}
                \expect[x]{h(\process_{s+t})}[\scrF_s] 
                & = \expect[x]{
                        \expect[\process_{s+t}]{
                            \1_{\set{\hittingtime{\absorbingstates}<\infty}}
                            }     
                    }[\scrF_s]\\ 
                & = \expect[x]{
                        \1_{\set{\hittingtime{\absorbingstates}<\infty} }\circ \theta_{s}
                    }[\scrF_s]\\ 
                & = \expect[\process_{s}]{
                        \1_{\set{\hittingtime{\absorbingstates}<\infty}}
                }\\ 
                & = h\br{X_s}.
            \end{align*}
            Thus, $h_{\alpha}\br{\process}$ is a supermartingale and $h\br{\process}$ is a martingale.
        \end{proof}
    
    \begin{lemma}
        \zcsetup{reftype=lemma}
        \label{lem:h-lsc}
        Let $\alpha$ be a positive number such that $\alpha\leq \min\br{1,\distortioncombi(\one)}$. 
        The function $h_{\alpha}$ defined by (\ref{def:h}) is lower semicontinuous.
    \end{lemma}
    \begin{proof}
        Recall that $\Omega$ is equipped with the topology of uniform convergence on compact subsets of $\reals_+$. 
        Let's first prove that $\omega\mapsto \1_{\set{\hittingtime{\absorbzone}<\infty}}(\omega)$ is lower semicontinuous on $\Omega$, \ie  for any $\gamma \in\reals$, the set $C_{\gamma} = \set{\omega \in \Omega,~\1_{\set{\hittingtime{\absorbzone}<\infty}}(\omega)\leq\gamma}$ is closed. 
        
        Fix $\gamma\in\reals$. If $\gamma\geq 1$, then $C_{\gamma} =\Omega$, and if $\gamma<0$, then $C_{\gamma} = \varnothing$. Suppose that $\gamma\in [0,1)$. Then 
        \[
        	C_{\gamma} = \set{\omega\in \Omega}[\1_{\set{\hittingtime{\absorbzone}<\infty}}(\omega) = 0] = \set{\omega\in \Omega}[\hittingtime{\absorbzone}(\omega)= \infty].
        \]
        Let $\br{\omega_n}_{n\in\nintegers}$ be a sequence of $C_{\gamma}$ converging uniformly on every compact subset of $\reals_+$. Fix $t\in\reals_+$. Since for any $n\in\nintegers$, $\hittingtime{\absorbzone}\br{\omega_n}=\infty$, in particular, $\br{\process_t\br{\omega_n}}_{n\in\nintegers}$ is a sequence of $\hypercube\setminus \absorbzone$, which is closed. Thus taking the limit, $\process_t(\omega)$ belongs to $\hypercube\setminus \absorbzone$. Consequently $\hittingtime{\absorbzone}(\omega) = \infty$, so $\omega\in C_{\gamma}$. Hence, $C_{\gamma}$ is closed.
        
        Now, let $\br{x_n}_{n\in\nintegers}$ be a sequence of $\hypercube$ that converges to $x\in \hypercube$. Consider the sequence of probability measures $\br{\proba[x_n]{}}$, where for any $n\in\nintegers$, $\proba[x_n]{}$ is the probability under which, almost surely, for any $t\in\reals_+$,
        \[
            \process_t = x_n +\int_0^t \dispersion\br{\process_s}\d \brownian_s+\int_0^t\drift\br{\process_s}\d s.
        \] 
        By \autocite[Theorem 17.25 of Chapter 17 in][386]{kallenberg2002foundations}, $\br{\proba[x_n]{}}$ converges weakly to $\proba[x]{}$ with respect to the topology of uniform convergence on all compact sets of $\reals_+$. By portmanteau theorem, as $\1_{\set{\hittingtime{\absorbzone}<\infty}}$ is bounded below and lower semicontinuous on $\Omega$,
        \[
            \liminf_{n\to\infty}\int_{\Omega}\1_{\set{\hittingtime{\absorbzone}<\infty }}\d \proba[x_n]{} \geq \int_{\Omega}\1_{\set{\hittingtime{\absorbzone}<\infty}}\d \proba[x]{},
        \]
        \ie
        \[
            \liminf_{n\to\infty} h_{\alpha}(x_n) \geq h_{\alpha}(x).
        \]
        In other words, the function $h_{\alpha}$ is lower semicontinuous.
    \end{proof}
    
    \begin{lemma}
        \zcsetup{reftype=lemma}
        \label{lem:h-positive}
        The function $h_{\alpha}$ defined by (\ref{def:h}) is constant and is equal to $1$.
    \end{lemma}
    
    \begin{proof}\leavevmode
        \begin{step}
            \item Let's prove by induction on $n\in\range{0,m}$ that for any $x\in \hypercube\setminus\absorbingstates$ such that $\card\br{\indices(x)}\leq n$, $h_{\alpha}(x)>0$ (see \zcref{not:cube}).
    
            \begin{itemize}
                \item $n=0$: Let $x\in\hypercube$ such that $\card\br{\indices(x)}=0$, meaning that $x\in\mathring{\hypercube}$. There exists an open connected subset $D$ of $\hypercube$ with $\continuous[\infty]{}$ boundary such that $D\cap \absorbzone \neq \varnothing$ and such that $\closure{D}\subset \mathring{\hypercube}$.
                \begin{figure}[ht]
                    \centering
                    \begin{tikzpicture}[scale=0.6,thick]
                        % sommets 
                        \coordinate (A) at (0,0);
                        \coordinate (B) at (10,0);
                        \coordinate (C) at (10,10);
                        \coordinate (D) at (0,10);

                        \fill[fill=colorDelta] (0,0)--(5,0)--(0,4);
                        \fill[fill=colorDelta] (10,10)--(5,10)--(10,6);
                        % côtés
                        \draw[line width=1pt] (A) -- (B);
                        \draw[line width=1pt] (B) -- (C);
                        \draw[line width=1pt] (C) -- (D);
                        \draw[line width=1pt] (D) -- (A);
                    
                        \draw[line width=1.5pt, color=colorDeltaBound] (0,0) -- (5,0) node[midway, left] {};
                        \draw[line width=1.5pt, color=colorDeltaBound] (0,0) -- (0,4) node[midway, left] {};
                    
                        \draw[line width=1.5pt, color=colorDeltaBound] (10,10) -- (5,10) node[midway, left] {};
                        \draw[line width=1.5pt, color=colorDeltaBound] (10,10) -- (10,6) node[midway, left] {};

                        \filldraw[colorDomain,fill opacity=0.5] (1.25 ,3) to [ curve through ={(1.55,2)  .. (4,0.8) .. (7,3)..(8,7) .. (5,8) }] (1.25,3);
                        
                        \draw[colorDomainBound] (1.25,3) to [ curve through ={(1.55,2)}](4,0.8) node[anchor=north east]{$\Gamma$};
            
                        % Points
                        \filldraw[color=colorAbsorbing] (A) circle (4pt) node[anchor=north east, colorAbsorbing]{$(0,0)$};
                        \filldraw[color=colorAbsorbing] (C) circle (4pt) node[anchor=south west, colorAbsorbing]{$(1,1)$};
                        \filldraw[color=colorInitialPoint] (4,6) circle (4pt) node[anchor=north west]{$x$};	
                    \end{tikzpicture}
                    \caption{Initial point $x$ with $\card{\indices(x)} = 0$ ($m=2$)}
                    \label{fig:ouvert-dans-interieur}
                \end{figure}
                We consider $\hittingtime{\partial D}$, the hitting time of the boundary $\partial D$ of $D$ by the process $\process$. Since the operator $\generator$ is uniformly elliptic on $D$, the Dirichlet problem 
                \[
                    \begin{cases}
                        \generator u  = -1 &\text{in }D,\\ 
                        u_{|\partial D} = 0 
                    \end{cases}
                \]
                has a unique smooth solution $u_D$, and by the maximum principle, $u_D$ is non-negative in $\closure{D}$. By Itô's rule, for all $x\in D$ and $t\in\reals_+$,
                \[
                    0\leq \expect[x]{u_D\br{t\wedge \hittingtime{\partial D}}} = u_D(x)-\expect[x]{t\wedge \hittingtime{\partial D}},
                \]
                and by monotone convergence, 
                \[
                    \expect[x]{\hittingtime{\partial D}}<\infty.
                \]
                As a result, $\hittingtime{\partial D}<\infty$ $\proba[x]{}$-a.s. Let us denote by $\Gamma$ the subset $\partial D\cap \absorbzone$, and by  $\Gamma^c$ the subset $\partial D \setminus \Gamma$ of $\partial D$. We have that $\hittingtime{\partial D} = \hittingtime{\Gamma}\wedge \hittingtime{\Gamma^c}$ and we want to prove that $\proba[x]{\hittingtime{\Gamma}<\hittingtime{\Gamma^c}}$ is positive.
                Let $\phi\in\continuous[\infty]{\partial D}$ be a $[0,1]$-valued function supported in $\Gamma$ such that there is a non-empty open subset $O$ of $\Gamma$ where $\phi_{|O} = 1$, so that $0\leq \phi\leq \1_{\Gamma}$. Then, the function $v_D$ defined by 
                \[
                    v_D(x) = \expect[x]{\phi\br{\process_{\hittingtime{\partial D}}}}
                \]
                is the unique smooth solution to the Dirichlet Problem
                \[
                \begin{cases}
                    \generator u = 0 & \text{in }D\\ 
                    u_{|\partial D} = \phi.
                \end{cases}
                \]
                Since the function $\phi$ is not identically zero, by the maximum principle, the function $v_D$ is positive in $D$. Now, for all $x\in D$,
                \[
                    \proba[x]{\hittingtime{\Gamma}<\hittingtime{\Gamma^c}} = \expect[x]{\1_{\set{\hittingtime{\Gamma}<\hittingtime{\Gamma^c}}}\1_{\Gamma}\br{\process_{\hittingtime{\partial D}}} }\geq \expect[x]{\1_{\set{\hittingtime{\Gamma}<\hittingtime{\Gamma^c}}} \phi\br{\process_{\hittingtime{\partial D}}}},
                \]
                and because $\phi= 0$ on $\Gamma^c$, 
                \[
                    \expect[x]{\1_{\set{\hittingtime{\Gamma}<\hittingtime{\Gamma^c}}} \phi\br{\process_{\hittingtime{\partial D}}}}= v_D(x)>0.
                \]
                Hence, 
                \[
                    h_{\alpha}(x) = \proba[x]{\hittingtime{\absorbzone}<\infty} \geq \proba[x]{\hittingtime{\Gamma}<\hittingtime{\Gamma^c}}>0.
                \]
                As $\mathring{K}$ is the reunion of all such subsets $D$, for any $x\in\mathring{\hypercube}$, $h_{\alpha}(x)>0$.
                \item Let $n\in\range{1,m-1}$ and suppose that for any $x\in \partial \hypercube\setminus\absorbingstates$ such that $\card\br{\indices(y)} \leq n$, $h_{\alpha}(y)>0$. Let $x\in \partial \hypercube$ be such that $\card{\indices(x)} = n+1$. Denote by $\face(x)$ the lower dimensional face that the point $x$ belongs to, \ie the set of the point $y\in\hypercube$ such that $y_i=0$ if $i\in\indices[0](x)$ and $y_i=1$ if $i\in\indices[1](x)$.
    
                The objective is to demonstrate that the process reaches a lower-dimensional face within a finite time. The recurrence hypothesis can then be applied to draw a conclusion. To do so, we construct a closed ball $B$ of $\reals^m$ such that any point $y$ of the set $B\cap\hypercube$ is either in the face $\face(x)$ or satisfies $\indices(y)\leq n$. The process at the hitting time of the boundary $\partial B$ will be in a lower-dimensional face if and only if it does not belong to the face $\face(x)$. 
                
                In order to characterise belonging to the face $\face(x)$, we introduce the function
                \[
                    \fun[N_x]{y}{\sum_{i\in\indices[0](x)}y_i + \sum_{i\in \indices[1](x)} \br{1-y_i}}[\hypercube][\reals_+],
                \]
                so that, for any $y\in\hypercube$, $N_x(y) = 0$ if and only if $y\in \face(x)$. Let us notice that $\nabla N_x$ is constant and equals $\normal(x)$, that the function $N_x$ belongs to $\continuouscube[2]$ and satisfies 
                \[
                    \generator N_x = \scalarproduct{b,\normal(x)},
                \] 
                and by \zcref{hyp:b-boundary-positive},
                \[
                    \generator N_x(x) >0.
                \]
                Therefore, by the continuity of $\generator N_x$, there exists $\delta>0$ such that 
                \[
                    0<\delta<\frac{1}{2}\min_{i\in\range{1,m}\setminus\indices(x)}\br{x_i,1-x_i}
                \]
                and that for any $y\in\hypercube$ such that $\abs{x-y}\leq\delta$, 
                \[
                    \generator N_x(y)\geq \generator N_x(x)/2.
                \]
                Set $B=B\br{x,\delta}$ the closed ball of $\reals^m$ centred at $x$ with radius $\delta$. Consider $\hittingtime{\partial B}$ the hitting time of $\partial B$ by the process $\process$. We first prove that the process hits $\partial B$ within a finite time almost surely. Then, using $N_x$, we prove that at the stopping time $\hittingtime{\partial B}$, there exists some positive constant $\kappa$ such that the distance between the process $\process_{\hittingtime{\partial B}}$ and the $\face(x)$ is larger than $\kappa$ with a positive probability.
                
                By Itô's rule, since $N_x(x)= 0$ and the martingale part is zero at time $t=0$, then, for any $t\in\reals_+$,
                \[
                    \expect[x]{N_x\br{\process_{t\wedge\hittingtime{\partial B}}}} = \expect[x]{\int_0^{t\wedge \hittingtime{\partial B}}\generator N_x\br{\process_s}\d s} \geq \frac{\generator N_x(x)}{2}\expect[x]{t\wedge \hittingtime{\partial B}}. 
                \]
                Hence, by monotone convergence,
                \begin{equation}
                    \zcsetup{reftype=inequality}
                    \label{ineq:ball}
                    \expect[x]{\hittingtime{\partial B}}\leq \expect[x]{N_x\br{\process_{\hittingtime{\partial B}}}} \leq \frac{2}{\generator N_x(x)}\normC{N_x},
                \end{equation}
                and therefore, $\proba[x]{\hittingtime{\partial B}<\infty} = 1$. 
                
                Moreover, since $\process$ is continuous, $\hittingtime{\partial B}>0$ $\proba[x]{}$-a.s. Yet, 
                \[
                    \set{\hittingtime{\partial B}>0} = \bigcup_{k\in\nintegers^*}\set{\hittingtime{\partial B}\geq \frac{1}{k}},
                \]
                hence,
                \[
                    1 = \proba[x]{\hittingtime{\partial B}>0} \leq \sum_{k=1}^{\infty}\proba[x]{\hittingtime{\partial B}\geq \frac{1}{k}}. 
                \]
                In particular, there exists $k\in\nintegers^*$ such that
                \[
                    \proba[x]{\hittingtime{\partial B}\geq\frac{1}{k}}>0,
                \]
                and as a result,
                \[
                    \expect[x]{\hittingtime{\partial B}}\geq \frac{1}{k} \proba[x]{\hittingtime{\partial B}\geq\frac{1}{k}}>0.
                \]
                By \zcref{ineq:ball},
                \[
                    \expect[x]{N_x\br{\process_{\hittingtime{\partial B}}}}>0.
                \]
                Moreover, denoting $d_{\face(x)}(y) = d\br{y,\face(x)} = \inf_{z\in \face(x)}\abs{y-z}$,
                \[
                \set{N_x\br{\process_{\hittingtime{\partial B}}}>0} = \set{d_{\face(x)}\br{\process_{\hittingtime{\partial B}}} >0} = \bigcup_{k=1}^{\infty}\set{d_{\face(x)}\br{\process_{\hittingtime{\partial B}}} \geq \frac{1}{k}}.
                \]
                In particular, there is a constant $\kappa>0$ such that
                \[
                    \proba[x]{d_{\face(x)}\br{\process_{\hittingtime{\partial B}}} \geq \kappa}>0,
                \]
                \ie 
                \[
					\proba[x]{\process_{\hittingtime{\partial B}}\in \Gamma}>0,
				\]
                where $\Gamma$ denotes the subset $\set{y\in \partial B\cap \hypercube}[d_{\face(x)}(y)\geq \kappa]$ of $\partial B\cap \hypercube$. For any $y\in\Gamma$, on one hand, since $\Gamma\subset \partial B\cap \hypercube$, then we have that for any $j\in\range{1,m}\setminus \indices(x)$, $y_j\notin\set{0,1}$, and on the other hand, since $d_{\face(x)}(y)>0$, there is $i\in\indices(x)$ such that $y_i\notin\set{0,1}$. Thus, for any $y\in \Gamma$, $\card{\indices(y)}\leq n$.
            
                \begin{figure}[ht]
                    \centering
                    \begin{tikzpicture}[scale=0.5]
    
                        \coordinate (A) at (0,0);
                        \coordinate (B) at (10,0);
                        \coordinate (C) at (10,10);
                        \coordinate (D) at (0,10);
                        
                        % Ball and Boundary
                        \filldraw[colorDomain, fill opacity=0.5] (0,4) arc(-90:90:2);
                        \draw[colorDomainBound,very thick] (2*cos{-70},6-2*sin{70}) arc (-70:70:2) node[xshift=10mm,yshift=-8mm] {$\Gamma$};
                        \draw [thick,dash dot] (2*cos{-70},6-2*sin{70}) -- (2*cos{70},6+2*sin{70});
                        \draw[<->] (0,5) -- (2*cos{70},5) node[below, xshift=-1.5mm,font=\scriptsize]{$\kappa$};

                        % Sides 
                        \draw[line width=1pt] (A) -- (B);
                        \draw[line width=1pt] (B) -- (C);
                        \draw[line width=1pt] (C) -- (D);
                        \draw[line width=1pt] (D) -- (A);
            
                        % Delta 
                        \fill[fill=colorDelta] (0,0)--(5,0)--(0,4);
                        \fill[fill=colorDelta] (10,10)--(5,10)--(10,6);
                        \draw[line width=1.5pt, color=colorDeltaBound] (0,0) -- (5,0) node[midway, left] {};
                        \draw[line width=1.5pt, color=colorDeltaBound] (0,0) -- (0,4) node[midway, left] {};
                        \draw[line width=1.5pt, color=colorDeltaBound] (10,10) -- (5,10) node[midway, left] {};
                        \draw[line width=1.5pt, color=colorDeltaBound] (10,10) -- (10,6) node[midway, left] {};
            
                        % Points
                        \filldraw[color=colorInitialPoint] (0,6) circle (4pt) node[anchor=east]{$x$};	
                        \filldraw[color=colorAbsorbing] (A) circle (4pt) node[anchor=north east, colorAbsorbing]{$(0,0)$};
                        \filldraw[color=colorAbsorbing] (C) circle (4pt) node[anchor=south west, colorAbsorbing]{$(1,1)$};
                    \end{tikzpicture}
                    \hfill
                    \begin{tikzpicture}[scale=0.5]
                            \coordinate (A) at (0,0);
                        \coordinate (B) at (10,0);
                        \coordinate (C) at (10,10);
                        \coordinate (D) at (0,10);
                        
                        % Ball 
                        \filldraw[colorDomain, fill opacity=0.5] (7,0) arc(180:90:3) --(B) -- cycle;
                        \draw[colorDomainBound,very thick] (7,0) arc(180:90:3) node[xshift=-12mm] {$\partial B$};
            
                         % Sides 
                        \draw[line width=1pt] (A) -- (B);
                        \draw[line width=1pt] (B) -- (C);
                        \draw[line width=1pt] (C) -- (D);
                        \draw[line width=1pt] (D) -- (A);
            
                        % Delta
                        \fill[fill=colorDelta] (0,0)--(5,0)--(0,4);
                        \fill[fill=colorDelta] (10,10)--(5,10)--(10,6);
                        
                        \draw[line width=1.5pt, color=colorDeltaBound] (0,0) -- (5,0) node[midway, left] {};
                        \draw[line width=1.5pt, color=colorDeltaBound] (0,0) -- (0,4) node[midway, left] {};
                    
                        \draw[line width=1.5pt, color=colorDeltaBound] (10,10) -- (5,10) node[midway, left] {};
                        \draw[line width=1.5pt, color=colorDeltaBound] (10,10) -- (10,6) node[midway, left] {};
                        
                        % Points 
                        \filldraw[color=colorInitialPoint] (10,0) circle (4pt) node[anchor=north west]{$x$};	
                        \filldraw[color=colorAbsorbing] (A) circle (4pt) node[anchor=north east, colorAbsorbing]{$(0,0)$};
                        \filldraw[color=colorAbsorbing] (C) circle (4pt) node[anchor=south west, colorAbsorbing]{$(1,1)$};
                    \end{tikzpicture}
                    \caption{Initial point $x$ with $\card{\indices(x)} = 1$ and $\card{\indices(x)} = m$ ($m=2$)}
                    \label{fig:boundary-1}
                \end{figure}
                Because $\Gamma=\partial B\cap \hypercube \cap {d_{\face(x)}}^{-1}\br{[\kappa,+\infty)}$, it is a compact subset of $\hypercube$. \zcref{lem:h-lsc} tells us that the function $h_{\alpha}$ is lower semicontinuous in $\hypercube$, and therefore in $\Gamma$. Thus, $\inf_{\Gamma}h_{\alpha}>0$.
                Since $\hittingtime{\partial B}$ is integrable, by the optional stopping time theorem, 
                \begin{align*}
                        h_{\alpha}(x) 
                        & \geq \expect[x]{
                            h_{\alpha}\br{\process_{\hittingtime{\partial B}}}
                        } \\
                        & \geq \expect[x]{
                            h_{\alpha}\br{\process_{\hittingtime{\partial B}}}\1_{\set{\process_{\hittingtime{\partial B}}\in \Gamma }}
                        } \\
                        & \geq \inf_{\Gamma}h_{\alpha}\proba[x]{\process_{\hittingtime{\partial B}\in \Gamma}}\\
                        & >0.
                \end{align*}
            \end{itemize}
            \item We are now able to demonstrate that the function $h_{\alpha}$ is constant and is equal to $1$.
            
            Since the function $h_{\alpha}$ is lower semicontinuous in the compact set $\hypercube$ and positive, there exists $x_{\min}\in \hypercube$ such that for all $x\in\hypercube$, $0<h_{\alpha}(x_{\min})\leq h_{\alpha}(x)$. By the continuity of $\process$, we have
            \[
                0<\proba[x_{\min}]{\hittingtime{\absorbzone}<\infty} = \proba[x_{\min}]{\exists t\in\rationals,~\process_t\in \absorbzone} \leq \sum_{t\in\rationals}\proba[x_{\min}]{\process_t\in\absorbzone}.
            \]
            In particular, there exists $t_0\in \reals_+$ such that 
            \[
                \proba[x_{\min}]{\process_{t_0}\in \absorbzone}>0.
            \]
            Again, since $h_{\alpha}\br{\process}$ is a supermartingale, then
            \begin{align*}
                h_{\alpha}(x_{\min}) 
                & \geq \expect[x_{\min}]{h_{\alpha}\br{\process_{t_0}}\1_{\set{\process_{t_0}\in \absorbzone}}}+ \expect[x_{\min}]{h_{\alpha}\br{\process_{t_0}}\br{1-\1_{\set{\process_{t_0}\in \absorbzone}}}}\\ 
                & \geq \proba[x_{\min}]{\process_{t_0}\in\absorbzone}+h_{\alpha}(x_{\min})\br{1-\proba[x_{\min}]{\process_{t_0}\in\absorbzone}},
            \end{align*}
            therefore, 
            \[
                h_{\alpha}(x_{\min})\geq 1.
            \]
        \end{step}
         Hence, $h_{\alpha}\equiv 1$.
    \end{proof}
    
    \begin{proof}[Proof of \texorpdfstring{\zcref{th:hitting-time}}]
        Set $\alpha=\min\br{1,\distortioncombi(\one)}$ and fix $x\in \absorbzone[\alpha/2]$. Since $\absorbzone[\alpha/2]\subset \absorbzone$, by \zcref{ineq:Delta-absorption-0} of \zcref{lem:Delta-absorption}, if $x\in\absorbzone^0$, then,
        \[
            \proba[x]{\hittingtime{\zero}<\infty}\geq 1-\frac{\distortioncombi(x)}{\alpha}\geq 1-\frac{\alpha/2}{\alpha}= \frac{1}{2},
        \]
        and likewise, by \zcref{ineq:Delta-absorption-1}, if $x\in\absorbzone^1$, then 
        \[
            \proba[x]{\hittingtime{\zero}<\infty}\geq 1-\frac{\alpha/2}{\alpha}= \frac{1}{2}.
        \]
        Thus, 
        \[
            \proba[x]{\hittingtime{\absorbingstates}<\infty}\geq \frac{1}{2}>0.
        \]
        For any $x\in \hypercube\setminus\absorbzone[\alpha/2]$, by \zcref{lem:h-positive}, $\proba[x]{\hittingtime{\absorbzone/2}<\infty} = 1$. Applying strong Markov property, 
        \begin{align*}
            \proba[x]{\hittingtime{\absorbingstates}<\infty} 
            & = \expect[x]{
                \1_{\set{\hittingtime{\absorbingstates}<\infty}}
            }\\ 
            & = \expect[x]{
                \expect[x]{
                    \1_{\set{\hittingtime{\absorbingstates}<\infty}}\circ \theta_{\hittingtime{\absorbzone[\alpha/2]}}
                }[\scrF_{\hittingtime{\absorbzone[\alpha/2]}}]
            }\\ 
            & = \expect[x]{
                \expect[\process_{\hittingtime{\absorbzone[\alpha/2]}}]{\1_{\set{\hittingtime{\absorbingstates}<\infty}}}
            }\\ 
            & = \expect[x]{
                \proba[\process_{\hittingtime{\absorbzone[\alpha/2]}}]{\hittingtime{\absorbingstates}<\infty}}\\
            & \geq \frac{1}{2}.\\
        \end{align*}
        Since the set $\absorbingstates$ is absorbing, the family of events  $\set{\process_n\in\absorbingstates}$ is increasing, so for any $x\in\hypercube$,
        \[
            0<\proba[x]{\hittingtime{\absorbingstates}<\infty} = \proba[x]{\exists n\in\nintegers,~\process_n\in\absorbingstates}=\proba[x]{\bigcup_{n\in\nintegers}\set{\process_n\in\absorbingstates}} = \lim_{n\to \infty}\proba[x]{\process_n\in \absorbingstates}.
        \]
        Yet, because the process $h(\process)$ is a martingale, for any $n\in\nintegers$,
        \begin{multline*}
        	h(x) = \expect[x]{h(\process_n)} =  \expect[x]{h\br{\process_{n}}\1_{\set{\process_{n}\in \absorbingstates}}}+ \expect[x]{h\br{\process_{n}}\br{1-\1_{\set{\process_{n}\in \absorbingstates}}}} \\
        	\geq \proba[x]{\process_{n}\in\absorbingstates}+\frac{1}{2}\br{1-\proba[x]{\process_{n}\in\absorbingstates}},
        \end{multline*}
        and as $n$ tends to infinity, 
        \[
            h(x) \geq h(x)+\frac{1}{2}\br{1-h(x)},
        \]
        \ie 
        \[
            \frac{1}{2}\br{1-h(x)}\leq 0,
        \]
        therefore, $h(x)\geq 1$.
        As a conclusion, for any $x\in\hypercube$, 
        \[
            \proba[x]{\hittingtime{\absorbingstates}<\infty} = 1.
        \]
    \end{proof}

\section*{Acknowledgements}

    The authors acknowledge the support of the CDP C\textsuperscript{2}EMPI, as well as the French State under the France-2030 programme, the University of Lille, the Initiative of Excellence of the University of Lille, the European Metropolis of Lille for their funding and support of the R-CDP-24-004-C\textsuperscript{2}EMPI Project. 
    
    \textsc{C. Wang} acknowledges the support of the BREF Programme Interne Blanc MITI 2023.2-CNRS
    and the PEPR-France 2030 Project Made in France n°ANR-25-PEFO-004, and thanks her supervisors \textsc{O. Goubet} and \textsc{F. Paccaut} for their help and the useful discussions.

\printbibliography 

\end{document}